\documentclass[12pt]{article}

\usepackage[usenames]{color}
\usepackage{amssymb}
\usepackage{amsmath}
\usepackage{amsthm}
\usepackage{amsfonts}
\usepackage{amscd}
\usepackage{graphicx}

\usepackage[colorlinks=true,
linkcolor=webgreen,
filecolor=webbrown,
citecolor=webgreen]{hyperref}

\definecolor{webgreen}{rgb}{0,.5,0}
\definecolor{webbrown}{rgb}{.6,0,0}

\usepackage{color}
\usepackage{fullpage}
\usepackage{float}

\usepackage{graphics}
\usepackage{latexsym}
\usepackage{epsf}

\newcommand{\seqnum}[1]{\href{https://oeis.org/#1}{\rm \underline{#1}}}

\newcommand{\Nset}{\mathbb{N}}
\newcommand{\Zset}{\mathbb{Z}}
\newcommand{\rb}[1]{_{\rm #1}}
\newcommand{\ch}[2]{\begin{bmatrix}#1\\ #2\end{bmatrix}}
\newcommand{\tch}[2]{[\begin{smallmatrix}#1\\ #2\end{smallmatrix}]}
\newcommand{\chb}[3]{\left\langle\begin{matrix}#1\\#2\end{matrix}\right\rangle_{\!\!\!#3}}
\newcommand{\tchb}[2]{\langle\begin{smallmatrix}#1\\#2\end{smallmatrix}\rangle}
\newcommand{\floor}[1]{\lfloor#1\rfloor}

\newlength{\defaultmultlinegap}
\newcommand{\zmlg}{\setlength{\multlinegap}{0pt}}
\newcommand{\rmlg}{\setlength{\multlinegap}{\defaultmultlinegap}}

\begin{document}

\theoremstyle{plain}
\newtheorem{theorem}{Theorem}
\newtheorem{corollary}[theorem]{Corollary}
\newtheorem{lemma}[theorem]{Lemma}
\newtheorem{proposition}[theorem]{Proposition}
\newtheorem{identity}[theorem]{Identity}

\theoremstyle{definition}
\newtheorem{definition}[theorem]{Definition}
\newtheorem{example}[theorem]{Example}
\newtheorem{conjecture}[theorem]{Conjecture}

\theoremstyle{remark}
\newtheorem{remark}[theorem]{Remark}

\begin{center}
\vskip 1cm{\LARGE\bf
On a Two-Parameter Family of   
\vskip .08in
Generalizations of Pascal's Triangle}
\vskip 1cm
\large
Michael A. Allen\\
Physics Department\\
Faculty of Science\\
Mahidol University\\
Rama 6 Road\\
Bangkok 10400  \\
Thailand \\
\href{mailto:maa5652@gmail.com}{\tt maa5652@gmail.com} \\
\end{center}

\vskip .2 in

\begin{abstract}
We consider a two-parameter family of triangles whose $(n,k)$-th entry
(counting the initial entry as the $(0,0)$-th entry) is the number of
tilings of $N$-boards (which are linear arrays of $N$ unit square cells for any
nonnegative integer $N$) with unit squares and $(1,m-1;t)$-combs for
some fixed $m=1,2,\dots$ and $t=2,3,\dots$ that use $n$ tiles in total
of which $k$ are combs. A $(1,m-1;t)$-comb is a tile composed of $t$
unit square sub-tiles (referred to as teeth) placed so that each tooth
is separated from the next by a gap of width $m-1$. We show that the
entries in the triangle are coefficients of the product of two
consecutive generalized Fibonacci polynomials each raised to some
nonnegative integer power. We also present a bijection between the
tiling of an $(n+(t-1)m)$-board with $k$ $(1,m-1;t)$-combs with the
remaining cells filled with squares and the $k$-subsets of
$\{1,\ldots,n\}$ such that no two elements of the subset differ by a
multiple of $m$ up to $(t-1)m$. We can therefore give a combinatorial
proof of how the number of such $k$-subsets is related to the
coefficient of a polynomial.  We also derive a recursion relation for
the number of closed walks from a particular node on a class of
directed pseudographs and apply it obtain an identity concerning the
$m=2$, $t=5$ instance of the family of triangles.  Further identities
of the triangles are also established mostly via combinatorial proof.
\end{abstract}

\section{Introduction}

In a recent paper \cite{AE22}, that we will henceforth refer to as
AE22, we considered two one-parameter families of generalizations of
Pascal's triangle.  Regarding the triangles as lower triangular
matrices, the members of both families have ones in the leftmost
column and the repetition of 1 followed by $m-1$ zeros along the
leading diagonal, where $m$ is a positive integer. In the case of the
first family, the rest of the entries are obtained using Pascal's
recurrence, i.e.,
$\tbinom{n}{k}_m=\tbinom{n-1}{k}_m+\tbinom{n-1}{k-1}_m$, where
$\tbinom{n}{k}_m$ is the $(n,k)$-th entry (counting the first entry as
being in row $n=0$ and column $k=0$) of the $m$-th triangle of the
family.  We showed that this is equivalent to the triangles being
row-reversed $(1/(1-x^m),x/(1-x))$ Riordan arrays. A
\textit{$(p(x),q(x))$ Riordan array} is an infinite lower triangular
matrix whose $(n,k)$-th entry is the coefficient of $x^n$ in the
series expansion of $p(x)(q(x))^k$ \cite{SGWW91,Bar=16}.  The
row-reversed version of a Riordan array has the entries up to and
including the leading diagonal in each row placed in reverse order
\cite{AE22}.

The main focus of AE22 was on a second family of triangles whose
$(n,k)$-th entry (denoted by $\tchb{n}{k}_m$) is the number of ways to
tile $N$-boards (which are linear arrays of $N\geq0$ unit square
cells) using $k$ $(1,m-1)$-fences and $n-k$ squares (and thus $n$
tiles in total).  A \textit{$(1,m-1)$-fence} is a tile composed to two
unit-square sub-tiles separated by a gap of width $m-1$
\cite{Edw08,EA15}.  The two families of triangles coincide for $m=1,2$
and the $m=1$ case is Pascal's triangle, i.e.,
$\tbinom{n}{k}_1=\tchb{n}{k}_1=\tbinom{n}{k}$ and
$\tbinom{n}{k}_2=\tchb{n}{k}_2$ for all $n$ and $k$.  We showed that
for $j\ge0$, $k\ge0$, $m\ge1$, and $r=0,\ldots,m-1$, the entry
$\tchb{mj+r-k}{k}_m$ is the coefficient of $x^k$ in
$f_j^{m-r}(x)f_{j+1}^r(x)$, where in this instance the Fibonacci
polynomial $f_n(x)$ is defined by
$f_n(x)=f_{n-1}(x)+xf_{n-2}(x)+\delta_{n,0}$, $f_{n<0}(x)=0$, where
$\delta_{i,j}$ is 1 if $i=j$ and zero otherwise. By first identifying
a bijection between the tilings of an $(n+m)$-board with $k$
$(1,m-1)$-fences and $n+m-2k$ squares and the subsets of
$\Nset_n=\{1,\ldots,n\}$ containing $k$ elements none of which differ
from another element in the subset by $m$, we showed that the number
of such subsets, $S^{(m)}(n,k)=\tchb{n+m-k}{k}_m$. We thus arrived at
a combinatorial proof of the relation between $S^{(m)}(n,k)$ and the
coefficient of $x^k$ in the product of nonnegative integer powers of
two successive Fibonacci polynomials.

Here we generalize the second family of triangles by considering the
analogous $n$-tile tilings of $N$-boards with $(1,m-1;t)$-combs and
squares for positive integer $m$ and $t=2,3,\ldots$. A
\textit{$(w,g;t)$-comb} contains $t$ sub-tiles of dimensions
$w\times1$ (referred to as \textit{teeth}) separated from one another
by gaps of width $g$
\cite{AE-GenFibSqr}. A
$(1,m-1;2)$-comb is evidently a $(1,m-1)$-fence and so the $t=2$
instances of the triangles we introduce here coincide with the second family of
triangles in AE22.

After introducing the two-parameter family of triangles (along with a
less compact version of the triangles which is used in some proofs) in
\S\ref{s:fam}, we show how entries of the triangles are related to
some generalized Fibonacci polynomials in \S\ref{s:poly}. Then in
\S\ref{s:comb} we give a bijection between the tilings of an
$(n+(t-1)m)$-board with $k$ $(1,m-1;t)$-combs and $n+(t-1)m-kt$
squares and the $k$-subsets of $\Nset_n$ such that no two
elements of the subset differ by any element of the set
$\{m,2m,\ldots,(t-1)m\}$. This enables us to relate the number of such
subsets to coefficients of products of powers of two successive
generalized Fibonacci polynomials. The remainder of the paper concerns
finding identities satisfied by entries in the triangle. Most of the
identities are obtained via the enumeration of metatiles with a
certain length or number of tiles which can be problematic if the
metatiles contain an arbitrary number of tiles. A \textit{metatile} is
a gapless grouping of tiles that completely covers a whole number of
cells and cannot be split into smaller metatiles. In most cases, there
are infinitely many possible metatiles and there have been various
approaches to the enumeration problem: obtaining the symbolic
representation all the families of metatiles \cite{EA19}, obtaining a
recursion relation for the number of metatiles of a certain length and
thus expressing the number in terms of a known sequence \cite{EA20a},
identifying a bijection between the metatiles and a set of objects
whose number is known \cite{AE-GenFibSqr}, and constructing a
directed pseudograph (that we refer to as a digraph) to represent the
placing of tiles \cite{EA15}. We will use the first and last of these
approaches and these are described further in
\S\ref{s:meta}. Recursion relations for numbers of tilings
corresponding to a particular class of digraph are derived in the
appendix and these are used to obtain identities for the $m=2$, $t=5$
triangle in \S\ref{s:idn} where further identities concerning the
triangles are also derived, mostly via combinatorial proof.

\section{The two-parameter family of triangles}\label{s:fam}

\begin{figure}[!b]
\begin{tabular}{c|*{14}{p{1.6em}}}
$n$ $\backslash$ $k$&\mbox{}\hfill0&\mbox{}\hfill1&\mbox{}\hfill2&\mbox{}\hfill3&\mbox{}\hfill4&\mbox{}\hfill5&\mbox{}\hfill6&\mbox{}\hfill7&\mbox{}\hfill8&\mbox{}\hfill9&\mbox{}\hfill10&\mbox{}\hfill11&\mbox{}\hfill12&\mbox{}\hfill13\\\hline
 0~~&\mbox{}\hspace*{\fill}\textbf{  1}\\
 1~~&\mbox{}\hspace*{\fill}\textbf{  1}&\mbox{}\hspace*{\fill}\textbf{  0}\\
 2~~&\mbox{}\hspace*{\fill}\textbf{  1}&\mbox{}\hspace*{\fill}\textbf{  0}&\mbox{}\hspace*{\fill}\textbf{  1}\\
 3~~&\mbox{}\hspace*{\fill}\textbf{  1}&\mbox{}\hspace*{\fill}\textbf{  1}&\mbox{}\hspace*{\fill}\textbf{  2}&\mbox{}\hspace*{\fill}\textbf{  0}\\
 4~~&\mbox{}\hspace*{\fill}\textbf{  1}&\mbox{}\hspace*{\fill}\textbf{  2}&\mbox{}\hspace*{\fill}\textbf{  4}&\mbox{}\hspace*{\fill}\textbf{  0}&\mbox{}\hspace*{\fill}\textbf{  1}\\
 5~~&\mbox{}\hspace*{\fill}\textbf{  1}&\mbox{}\hspace*{\fill}\textbf{  3}&\mbox{}\hspace*{\fill}\textbf{  6}&\mbox{}\hspace*{\fill}\textbf{  3}&\mbox{}\hspace*{\fill}\textbf{  3}&\mbox{}\hspace*{\fill}\textbf{  0}\\
 6~~&\mbox{}\hspace*{\fill}\textbf{  1}&\mbox{}\hspace*{\fill}\textbf{  4}&\mbox{}\hspace*{\fill}\textbf{  9}&\mbox{}\hfill  8 &\mbox{}\hspace*{\fill}\textbf{  9}&\mbox{}\hspace*{\fill}\textbf{  0}&\mbox{}\hspace*{\fill}\textbf{  1}\\
 7~~&\mbox{}\hspace*{\fill}\textbf{  1}&\mbox{}\hspace*{\fill}\textbf{  5}&\mbox{}\hfill 13 &\mbox{}\hfill 17 &\mbox{}\hspace*{\fill}\textbf{ 18}&\mbox{}\hspace*{\fill}\textbf{  6}&\mbox{}\hspace*{\fill}\textbf{  4}&\mbox{}\hspace*{\fill}\textbf{  0}\\
 8~~&\mbox{}\hspace*{\fill}\textbf{  1}&\mbox{}\hspace*{\fill}\textbf{  6}&\mbox{}\hfill 18 &\mbox{}\hfill 30 &\mbox{}\hspace*{\fill}\textbf{ 36}&\mbox{}\hfill 20 &\mbox{}\hspace*{\fill}\textbf{ 16}&\mbox{}\hspace*{\fill}\textbf{  0}&\mbox{}\hspace*{\fill}\textbf{  1}\\
 9~~&\mbox{}\hspace*{\fill}\textbf{  1}&\mbox{}\hspace*{\fill}\textbf{  7}&\mbox{}\hfill 24 &\mbox{}\hfill 48 &\mbox{}\hfill 66 &\mbox{}\hfill 55 &\mbox{}\hspace*{\fill}\textbf{ 40}&\mbox{}\hspace*{\fill}\textbf{ 10}&\mbox{}\hspace*{\fill}\textbf{  5}&\mbox{}\hspace*{\fill}\textbf{  0}\\
10~~&\mbox{}\hspace*{\fill}\textbf{  1}&\mbox{}\hspace*{\fill}\textbf{  8}&\mbox{}\hfill 31 &\mbox{}\hfill 72 &\mbox{}\hfill114 &\mbox{}\hfill120 &\mbox{}\hspace*{\fill}\textbf{100}&\mbox{}\hfill 40 &\mbox{}\hspace*{\fill}\textbf{ 25}&\mbox{}\hspace*{\fill}\textbf{  0}&\mbox{}\hspace*{\fill}\textbf{  1}\\
11~~&\mbox{}\hspace*{\fill}\textbf{  1}&\mbox{}\hspace*{\fill}\textbf{  9}&\mbox{}\hfill 39 &\mbox{}\hfill103 &\mbox{}\hfill186 &\mbox{}\hfill234 &\mbox{}\hfill221 &\mbox{}\hfill135 &\mbox{}\hspace*{\fill}\textbf{ 75}&\mbox{}\hspace*{\fill}\textbf{ 15}&\mbox{}\hspace*{\fill}\textbf{  6}&\mbox{}\hspace*{\fill}\textbf{  0}\\
12~~&\mbox{}\hspace*{\fill}\textbf{  1}&\mbox{}\hspace*{\fill}\textbf{ 10}&\mbox{}\hfill 48 &\mbox{}\hfill142 &\mbox{}\hfill289 &\mbox{}\hfill420 &\mbox{}\hfill456 &\mbox{}\hfill350 &\mbox{}\hspace*{\fill}\textbf{225}&\mbox{}\hfill 70 &\mbox{}\hspace*{\fill}\textbf{ 36}&\mbox{}\hspace*{\fill}\textbf{  0}&\mbox{}\hspace*{\fill}\textbf{  1}\\
13~~&\mbox{}\hspace*{\fill}\textbf{  1}&\mbox{}\hspace*{\fill}\textbf{ 11}&\mbox{}\hfill 58 &\mbox{}\hfill190 &\mbox{}\hfill431 &\mbox{}\hfill709 &\mbox{}\hfill876 &\mbox{}\hfill805 &\mbox{}\hfill581 &\mbox{}\hfill280 &\mbox{}\hspace*{\fill}\textbf{126}&\mbox{}\hspace*{\fill}\textbf{ 21}&\mbox{}\hspace*{\fill}\textbf{  7}&\mbox{}\hspace*{\fill}\textbf{  0}\\
\end{tabular}
\caption{The start of a Pascal-like triangle (\seqnum{A354665} in the
  OEIS \cite{Slo-OEIS}) whose
  $(n,k)$-th entry, $\protect\tchb{n}{k}_{2,3}$, is
  the number of $n$-tile tilings using $k$
  $(1,1;3)$-combs (and $n-k$ squares). Entries in bold font (and those
  in bold font in Figs.~\ref{f:m=2,t=4}--\ref{f:m=3,t=3}) are covered by
  identities in \S\ref{s:idn}.}
\label{f:m=2,t=3}
\end{figure}

For $m=1,2,\ldots$ and $t=2,3,\ldots$, let $\tchb{n}{k}_{m,t}$ denote
the number of $n$-tile tilings of $N$-boards that use $k$ $(1,m-1;t)$-combs (and
$n-k$ squares).  We choose that $\tchb{0}{0}_{m,t}=1$ and that
$\tchb{n}{k<0}_{m,t}=\tchb{n}{k>n}_{m,t}=0$.
As a $(1,0;t)$-comb is just a $t$-omino and the
number of $n$-tile tilings using $n$ $t$-ominoes and $n-k$ squares is
 simply $\tbinom{n}{k}$ for any $t$, we have
$\tchb{n}{k}_{1,t\ge2}=\tbinom{n}{k}$ which is Pascal's triangle
(\seqnum{A007318}).  The triangles corresponding to $m=2,3,4,5$ with
$t=2$ are \seqnum{A059259}, \seqnum{A350110}, \seqnum{A350111}, and
\seqnum{A350112}, respectively \cite{AE22}.  We show examples
of the starts of triangles for combs with at least 3 teeth in
Figs.~\ref{f:m=2,t=3}--\ref{f:m=3,t=3}.

\begin{figure}
\begin{tabular}{c|*{14}{p{1.6em}}}
$n$ $\backslash$ $k$&\mbox{}\hfill0&\mbox{}\hfill1&\mbox{}\hfill2&\mbox{}\hfill3&\mbox{}\hfill4&\mbox{}\hfill5&\mbox{}\hfill6&\mbox{}\hfill7&\mbox{}\hfill8&\mbox{}\hfill9&\mbox{}\hfill10&\mbox{}\hfill11&\mbox{}\hfill12&\mbox{}\hfill13\\\hline
 0~~&\mbox{}\hspace*{\fill}\textbf{  1}\\
 1~~&\mbox{}\hspace*{\fill}\textbf{  1}&\mbox{}\hspace*{\fill}\textbf{  0}\\
 2~~&\mbox{}\hspace*{\fill}\textbf{  1}&\mbox{}\hspace*{\fill}\textbf{  0}&\mbox{}\hspace*{\fill}\textbf{  1}\\
 3~~&\mbox{}\hspace*{\fill}\textbf{  1}&\mbox{}\hspace*{\fill}\textbf{  0}&\mbox{}\hspace*{\fill}\textbf{  2}&\mbox{}\hspace*{\fill}\textbf{  0}\\
 4~~&\mbox{}\hspace*{\fill}\textbf{  1}&\mbox{}\hspace*{\fill}\textbf{  1}&\mbox{}\hspace*{\fill}\textbf{  4}&\mbox{}\hspace*{\fill}\textbf{  0}&\mbox{}\hspace*{\fill}\textbf{  1}\\
 5~~&\mbox{}\hspace*{\fill}\textbf{  1}&\mbox{}\hspace*{\fill}\textbf{  2}&\mbox{}\hspace*{\fill}\textbf{  6}&\mbox{}\hspace*{\fill}\textbf{  0}&\mbox{}\hspace*{\fill}\textbf{  3}&\mbox{}\hspace*{\fill}\textbf{  0}\\
 6~~&\mbox{}\hspace*{\fill}\textbf{  1}&\mbox{}\hspace*{\fill}\textbf{  3}&\mbox{}\hspace*{\fill}\textbf{  9}&\mbox{}\hspace*{\fill}\textbf{  4}&\mbox{}\hspace*{\fill}\textbf{  9}&\mbox{}\hspace*{\fill}\textbf{  0}&\mbox{}\hspace*{\fill}\textbf{  1}\\
 7~~&\mbox{}\hspace*{\fill}\textbf{  1}&\mbox{}\hspace*{\fill}\textbf{  4}&\mbox{}\hspace*{\fill}\textbf{ 12}&\mbox{}\hfill 10 &\mbox{}\hspace*{\fill}\textbf{ 18}&\mbox{}\hspace*{\fill}\textbf{  0}&\mbox{}\hspace*{\fill}\textbf{  4}&\mbox{}\hspace*{\fill}\textbf{  0}\\
 8~~&\mbox{}\hspace*{\fill}\textbf{  1}&\mbox{}\hspace*{\fill}\textbf{  5}&\mbox{}\hspace*{\fill}\textbf{ 16}&\mbox{}\hfill 21 &\mbox{}\hspace*{\fill}\textbf{ 36}&\mbox{}\hspace*{\fill}\textbf{ 10}&\mbox{}\hspace*{\fill}\textbf{ 16}&\mbox{}\hspace*{\fill}\textbf{  0}&\mbox{}\hspace*{\fill}\textbf{  1}\\
 9~~&\mbox{}\hspace*{\fill}\textbf{  1}&\mbox{}\hspace*{\fill}\textbf{  6}&\mbox{}\hfill 21 &\mbox{}\hfill 36 &\mbox{}\hspace*{\fill}\textbf{ 60}&\mbox{}\hfill 30 &\mbox{}\hspace*{\fill}\textbf{ 40}&\mbox{}\hspace*{\fill}\textbf{  0}&\mbox{}\hspace*{\fill}\textbf{  5}&\mbox{}\hspace*{\fill}\textbf{  0}\\
10~~&\mbox{}\hspace*{\fill}\textbf{  1}&\mbox{}\hspace*{\fill}\textbf{  7}&\mbox{}\hfill 27 &\mbox{}\hfill 57 &\mbox{}\hspace*{\fill}\textbf{100}&\mbox{}\hfill 81 &\mbox{}\hspace*{\fill}\textbf{100}&\mbox{}\hspace*{\fill}\textbf{ 20}&\mbox{}\hspace*{\fill}\textbf{ 25}&\mbox{}\hspace*{\fill}\textbf{  0}&\mbox{}\hspace*{\fill}\textbf{  1}\\
11~~&\mbox{}\hspace*{\fill}\textbf{  1}&\mbox{}\hspace*{\fill}\textbf{  8}&\mbox{}\hfill 34 &\mbox{}\hfill 84 &\mbox{}\hfill158 &\mbox{}\hfill168 &\mbox{}\hspace*{\fill}\textbf{200}&\mbox{}\hfill 70 &\mbox{}\hspace*{\fill}\textbf{ 75}&\mbox{}\hspace*{\fill}\textbf{  0}&\mbox{}\hspace*{\fill}\textbf{  6}&\mbox{}\hspace*{\fill}\textbf{  0}\\
12~~&\mbox{}\hspace*{\fill}\textbf{  1}&\mbox{}\hspace*{\fill}\textbf{  9}&\mbox{}\hfill 42 &\mbox{}\hfill118 &\mbox{}\hfill243 &\mbox{}\hfill322 &\mbox{}\hspace*{\fill}\textbf{400}&\mbox{}\hfill231 &\mbox{}\hspace*{\fill}\textbf{225}&\mbox{}\hspace*{\fill}\textbf{ 35}&\mbox{}\hspace*{\fill}\textbf{ 36}&\mbox{}\hspace*{\fill}\textbf{  0}&\mbox{}\hspace*{\fill}\textbf{  1}\\
13~~&\mbox{}\hspace*{\fill}\textbf{  1}&\mbox{}\hspace*{\fill}\textbf{ 10}&\mbox{}\hfill 51 &\mbox{}\hfill160 &\mbox{}\hfill361 &\mbox{}\hfill560 &\mbox{}\hfill736 &\mbox{}\hfill560 &\mbox{}\hspace*{\fill}\textbf{525}&\mbox{}\hfill140 &\mbox{}\hspace*{\fill}\textbf{126}&\mbox{}\hspace*{\fill}\textbf{  0}&\mbox{}\hspace*{\fill}\textbf{  7}&\mbox{}\hspace*{\fill}\textbf{  0}\\
\end{tabular}
\caption{The start of a Pascal-like triangle (\seqnum{A354666}) with
  entries $\protect\tchb{n}{k}_{2,4}$.}
\label{f:m=2,t=4}
\end{figure}

\begin{figure}
\begin{tabular}{c|*{14}{p{1.6em}}}
$n$ $\backslash$ $k$&\mbox{}\hfill0&\mbox{}\hfill1&\mbox{}\hfill2&\mbox{}\hfill3&\mbox{}\hfill4&\mbox{}\hfill5&\mbox{}\hfill6&\mbox{}\hfill7&\mbox{}\hfill8&\mbox{}\hfill9&\mbox{}\hfill10&\mbox{}\hfill11&\mbox{}\hfill12&\mbox{}\hfill13\\\hline
 0~~&\mbox{}\hspace*{\fill}\textbf{  1}\\
 1~~&\mbox{}\hspace*{\fill}\textbf{  1}&\mbox{}\hspace*{\fill}\textbf{  0}\\
 2~~&\mbox{}\hspace*{\fill}\textbf{  1}&\mbox{}\hspace*{\fill}\textbf{  0}&\mbox{}\hspace*{\fill}\textbf{  1}\\
 3~~&\mbox{}\hspace*{\fill}\textbf{  1}&\mbox{}\hspace*{\fill}\textbf{  0}&\mbox{}\hspace*{\fill}\textbf{  2}&\mbox{}\hspace*{\fill}\textbf{  0}\\
 4~~&\mbox{}\hspace*{\fill}\textbf{  1}&\mbox{}\hspace*{\fill}\textbf{  0}&\mbox{}\hspace*{\fill}\textbf{  4}&\mbox{}\hspace*{\fill}\textbf{  0}&\mbox{}\hspace*{\fill}\textbf{  1}\\
 5~~&\mbox{}\hspace*{\fill}\textbf{  1}&\mbox{}\hspace*{\fill}\textbf{  1}&\mbox{}\hspace*{\fill}\textbf{  6}&\mbox{}\hspace*{\fill}\textbf{  0}&\mbox{}\hspace*{\fill}\textbf{  3}&\mbox{}\hspace*{\fill}\textbf{  0}\\
 6~~&\mbox{}\hspace*{\fill}\textbf{  1}&\mbox{}\hspace*{\fill}\textbf{  2}&\mbox{}\hspace*{\fill}\textbf{  9}&\mbox{}\hspace*{\fill}\textbf{  0}&\mbox{}\hspace*{\fill}\textbf{  9}&\mbox{}\hspace*{\fill}\textbf{  0}&\mbox{}\hspace*{\fill}\textbf{  1}\\
 7~~&\mbox{}\hspace*{\fill}\textbf{  1}&\mbox{}\hspace*{\fill}\textbf{  3}&\mbox{}\hspace*{\fill}\textbf{ 12}&\mbox{}\hspace*{\fill}\textbf{  5}&\mbox{}\hspace*{\fill}\textbf{ 18}&\mbox{}\hspace*{\fill}\textbf{  0}&\mbox{}\hspace*{\fill}\textbf{  4}&\mbox{}\hspace*{\fill}\textbf{  0}\\
 8~~&\mbox{}\hspace*{\fill}\textbf{  1}&\mbox{}\hspace*{\fill}\textbf{  4}&\mbox{}\hspace*{\fill}\textbf{ 16}&\mbox{}\hfill 12 &\mbox{}\hspace*{\fill}\textbf{ 36}&\mbox{}\hspace*{\fill}\textbf{  0}&\mbox{}\hspace*{\fill}\textbf{ 16}&\mbox{}\hspace*{\fill}\textbf{  0}&\mbox{}\hspace*{\fill}\textbf{  1}\\
 9~~&\mbox{}\hspace*{\fill}\textbf{  1}&\mbox{}\hspace*{\fill}\textbf{  5}&\mbox{}\hspace*{\fill}\textbf{ 20}&\mbox{}\hfill 25 &\mbox{}\hspace*{\fill}\textbf{ 60}&\mbox{}\hspace*{\fill}\textbf{ 15}&\mbox{}\hspace*{\fill}\textbf{ 40}&\mbox{}\hspace*{\fill}\textbf{  0}&\mbox{}\hspace*{\fill}\textbf{  5}&\mbox{}\hspace*{\fill}\textbf{  0}\\
10~~&\mbox{}\hspace*{\fill}\textbf{  1}&\mbox{}\hspace*{\fill}\textbf{  6}&\mbox{}\hspace*{\fill}\textbf{ 25}&\mbox{}\hfill 42 &\mbox{}\hspace*{\fill}\textbf{100}&\mbox{}\hfill 42 &\mbox{}\hspace*{\fill}\textbf{100}&\mbox{}\hspace*{\fill}\textbf{  0}&\mbox{}\hspace*{\fill}\textbf{ 25}&\mbox{}\hspace*{\fill}\textbf{  0}&\mbox{}\hspace*{\fill}\textbf{  1}\\
11~~&\mbox{}\hspace*{\fill}\textbf{  1}&\mbox{}\hspace*{\fill}\textbf{  7}&\mbox{}\hfill 31 &\mbox{}\hfill 66 &\mbox{}\hspace*{\fill}\textbf{150}&\mbox{}\hfill112 &\mbox{}\hspace*{\fill}\textbf{200}&\mbox{}\hspace*{\fill}\textbf{ 35}&\mbox{}\hspace*{\fill}\textbf{ 75}&\mbox{}\hspace*{\fill}\textbf{  0}&\mbox{}\hspace*{\fill}\textbf{  6}&\mbox{}\hspace*{\fill}\textbf{  0}\\
12~~&\mbox{}\hspace*{\fill}\textbf{  1}&\mbox{}\hspace*{\fill}\textbf{  8}&\mbox{}\hfill 38 &\mbox{}\hfill 96 &\mbox{}\hspace*{\fill}\textbf{225}&\mbox{}\hfill224 &\mbox{}\hspace*{\fill}\textbf{400}&\mbox{}\hfill112 &\mbox{}\hspace*{\fill}\textbf{225}&\mbox{}\hspace*{\fill}\textbf{  0}&\mbox{}\hspace*{\fill}\textbf{ 36}&\mbox{}\hspace*{\fill}\textbf{  0}&\mbox{}\hspace*{\fill}\textbf{  1}\\
13~~&\mbox{}\hspace*{\fill}\textbf{  1}&\mbox{}\hspace*{\fill}\textbf{  9}&\mbox{}\hfill 46 &\mbox{}\hfill134 &\mbox{}\hfill325 &\mbox{}\hfill424 &\mbox{}\hspace*{\fill}\textbf{700}&\mbox{}\hfill364 &\mbox{}\hspace*{\fill}\textbf{525}&\mbox{}\hspace*{\fill}\textbf{ 70}&\mbox{}\hspace*{\fill}\textbf{126}&\mbox{}\hspace*{\fill}\textbf{  0}&\mbox{}\hspace*{\fill}\textbf{  7}&\mbox{}\hspace*{\fill}\textbf{  0}\\
\end{tabular}
\caption{The start of a Pascal-like triangle (\seqnum{A354667}) with
  entries $\protect\tchb{n}{k}_{2,5}$.}
\label{f:m=2,t=5}
\end{figure}

\begin{figure}
\begin{tabular}{c|*{14}{p{1.6em}}}
$n$ $\backslash$ $k$&\mbox{}\hfill0&\mbox{}\hfill1&\mbox{}\hfill2&\mbox{}\hfill3&\mbox{}\hfill4&\mbox{}\hfill5&\mbox{}\hfill6&\mbox{}\hfill7&\mbox{}\hfill8&\mbox{}\hfill9&\mbox{}\hfill10&\mbox{}\hfill11&\mbox{}\hfill12&\mbox{}\hfill13\\\hline
 0~~&\mbox{}\hspace*{\fill}\textbf{  1}\\
 1~~&\mbox{}\hspace*{\fill}\textbf{  1}&\mbox{}\hspace*{\fill}\textbf{  0}\\
 2~~&\mbox{}\hspace*{\fill}\textbf{  1}&\mbox{}\hspace*{\fill}\textbf{  0}&\mbox{}\hspace*{\fill}\textbf{  0}\\
 3~~&\mbox{}\hspace*{\fill}\textbf{  1}&\mbox{}\hspace*{\fill}\textbf{  0}&\mbox{}\hspace*{\fill}\textbf{  0}&\mbox{}\hspace*{\fill}\textbf{  1}\\
 4~~&\mbox{}\hspace*{\fill}\textbf{  1}&\mbox{}\hspace*{\fill}\textbf{  0}&\mbox{}\hspace*{\fill}\textbf{  1}&\mbox{}\hspace*{\fill}\textbf{  2}&\mbox{}\hspace*{\fill}\textbf{  0}\\
 5~~&\mbox{}\hspace*{\fill}\textbf{  1}&\mbox{}\hspace*{\fill}\textbf{  1}&\mbox{}\hspace*{\fill}\textbf{  3}&\mbox{}\hspace*{\fill}\textbf{  4}&\mbox{}\hspace*{\fill}\textbf{  0}&\mbox{}\hspace*{\fill}\textbf{  0}\\
 6~~&\mbox{}\hspace*{\fill}\textbf{  1}&\mbox{}\hspace*{\fill}\textbf{  2}&\mbox{}\hfill  5 &\mbox{}\hspace*{\fill}\textbf{  8}&\mbox{}\hspace*{\fill}\textbf{  0}&\mbox{}\hspace*{\fill}\textbf{  0}&\mbox{}\hspace*{\fill}\textbf{  1}\\
 7~~&\mbox{}\hspace*{\fill}\textbf{  1}&\mbox{}\hspace*{\fill}\textbf{  3}&\mbox{}\hfill  8 &\mbox{}\hspace*{\fill}\textbf{ 12}&\mbox{}\hspace*{\fill}\textbf{  0}&\mbox{}\hspace*{\fill}\textbf{  3}&\mbox{}\hspace*{\fill}\textbf{  3}&\mbox{}\hspace*{\fill}\textbf{  0}\\
 8~~&\mbox{}\hspace*{\fill}\textbf{  1}&\mbox{}\hspace*{\fill}\textbf{  4}&\mbox{}\hfill 12 &\mbox{}\hspace*{\fill}\textbf{ 18}&\mbox{}\hspace*{\fill}\textbf{  9}&\mbox{}\hspace*{\fill}\textbf{ 12}&\mbox{}\hspace*{\fill}\textbf{  9}&\mbox{}\hspace*{\fill}\textbf{  0}&\mbox{}\hspace*{\fill}\textbf{  0}\\
 9~~&\mbox{}\hspace*{\fill}\textbf{  1}&\mbox{}\hspace*{\fill}\textbf{  5}&\mbox{}\hfill 16 &\mbox{}\hspace*{\fill}\textbf{ 27}&\mbox{}\hfill 25 &\mbox{}\hfill 29 &\mbox{}\hspace*{\fill}\textbf{ 27}&\mbox{}\hspace*{\fill}\textbf{  0}&\mbox{}\hspace*{\fill}\textbf{  0}&\mbox{}\hspace*{\fill}\textbf{  1}\\
10~~&\mbox{}\hspace*{\fill}\textbf{  1}&\mbox{}\hspace*{\fill}\textbf{  6}&\mbox{}\hfill 21 &\mbox{}\hfill 42 &\mbox{}\hfill 51 &\mbox{}\hfill 66 &\mbox{}\hspace*{\fill}\textbf{ 54}&\mbox{}\hspace*{\fill}\textbf{  0}&\mbox{}\hspace*{\fill}\textbf{  6}&\mbox{}\hspace*{\fill}\textbf{  4}&\mbox{}\hspace*{\fill}\textbf{  0}\\
11~~&\mbox{}\hspace*{\fill}\textbf{  1}&\mbox{}\hspace*{\fill}\textbf{  7}&\mbox{}\hfill 27 &\mbox{}\hfill 62 &\mbox{}\hfill 95 &\mbox{}\hfill135 &\mbox{}\hspace*{\fill}\textbf{108}&\mbox{}\hspace*{\fill}\textbf{ 36}&\mbox{}\hspace*{\fill}\textbf{ 30}&\mbox{}\hspace*{\fill}\textbf{ 16}&\mbox{}\hspace*{\fill}\textbf{  0}&\mbox{}\hspace*{\fill}\textbf{  0}\\
12~~&\mbox{}\hspace*{\fill}\textbf{  1}&\mbox{}\hspace*{\fill}\textbf{  8}&\mbox{}\hfill 34 &\mbox{}\hfill 88 &\mbox{}\hfill160 &\mbox{}\hfill234 &\mbox{}\hspace*{\fill}\textbf{216}&\mbox{}\hfill126 &\mbox{}\hfill 95 &\mbox{}\hspace*{\fill}\textbf{ 64}&\mbox{}\hspace*{\fill}\textbf{  0}&\mbox{}\hspace*{\fill}\textbf{  0}&\mbox{}\hspace*{\fill}\textbf{  1}\\
13~~&\mbox{}\hspace*{\fill}\textbf{  1}&\mbox{}\hspace*{\fill}\textbf{  9}&\mbox{}\hfill 42 &\mbox{}\hfill122 &\mbox{}\hfill252 &\mbox{}\hfill396 &\mbox{}\hfill432 &\mbox{}\hfill321 &\mbox{}\hfill280 &\mbox{}\hspace*{\fill}\textbf{160}&\mbox{}\hspace*{\fill}\textbf{  0}&\mbox{}\hspace*{\fill}\textbf{ 10}&\mbox{}\hspace*{\fill}\textbf{  5}&\mbox{}\hspace*{\fill}\textbf{  0}\\
\end{tabular}
\caption{The start of a Pascal-like triangle (\seqnum{A354668}) with
  entries $\protect\tchb{n}{k}_{3,3}$.}
\label{f:m=3,t=3}
\end{figure}

We can also create a triangle of $\tch{n}{k}_{m,t}$ where this denotes
the number of tilings of an $n$-board that use $k$ $(1,m-1;t)$-combs
(and therefore $n-kt$ squares) again with $\tch{0}{0}_{m,t}=1$. The
two triangles are related via the following identity.

\begin{identity}\label{I:ch=chb}
For $m\ge1$, $t\ge2$, and $n\ge k\ge0$, 
\[
\ch{n}{k}_{m,t}=\chb{n-(t-1)k}{k}{m,t}.
\]
\end{identity}
\begin{proof}
If a tiling contains $n-(t-1)k$ tiles of which $k$ are
$(1,m-1;t)$-combs (and so $n-kt$ are squares), the total length is $n-kt+kt=n$.
\end{proof}

We will refer to the ray of entries given by $\tchb{n-\mu k}{k}_{m,t}$ for
$k=0,\ldots,\floor{n/(\mu+1)}$ as the $n$-th
\textit{$(1,\mu)$-antidiagonal}. A $(1,1)$-antidiagonal is therefore
what is normally referred to simply as an antidiagonal. 
As a consequence of Identity~\ref{I:ch=chb}, the
$(1,t-1)$-antidiagonals of the $\tchb{n}{k}_{m,t}$ triangle are the
rows of the $\tch{n}{k}_{m,t}$ triangle.  In the rest of the paper we
therefore only give identities for the $\tchb{n}{k}_{m,t}$ triangle as
it is more `compact' in the sense that its rows contain fewer trailing
zeros. However, as in AE22, some of the identities are more
straightforward to prove by considering the tiling of an $n$-board, in
which case we need to consider $\tch{n}{k}_{m,t}$.  The following
bijection (which is established in the proof of Theorem~2.1 in
\cite{AE-GenFibSqr}) will be used in such proofs.

\begin{lemma}\label{L:bij} 
For $t\ge2$, $j\ge0$, and $r=0,\ldots,m$, where $m\ge1$, there is a
bijection between the tilings of an $(mj+r)$-board using $k$
$(1,m-1;t)$-combs and $mj+r-kt$ squares and the tilings of an ordered
$m$-tuple of $r$ $(j+1)$-boards followed by $m-r$ $j$-boards using $k$
$t$-ominoes and $mj+r-kt$ squares.
\end{lemma}

\section{Relation of the triangles to polynomials}\label{s:poly}

For $t\geq2$ we define a $(1,t)$-bonacci polynomial as follows:
\begin{equation}\label{e:f(x)}
f^{(t)}_n(x)=f^{(t)}_{n-1}(x)+xf^{(t)}_{n-t}(x)+\delta_{n,0}, 
\quad f^{(t)}_{n<0}(x)=0.
\end{equation}
The $(1,2)$-bonacci polynomials $f^{(2)}_n(x)$ are the Fibonacci
polynomials used in AE22. We refer to the sequence defined by 
\begin{equation}\label{e:f}
f^{(t)}_n=f^{(t)}_{n-1}+f^{(t)}_{n-t}+\delta_{n,0}, 
\quad f^{(t)}_{n<0}=0,
\end{equation}
for $t\geq2$ as the $(1,t)$-bonacci numbers. The $t=2,\ldots,8$ cases
are, respectively, the Fibonacci numbers (\seqnum{A000045}), the
Narayana's cows sequence (\seqnum{A000930}) and sequences
\seqnum{A003269}, \seqnum{A003520}, \seqnum{A005708}, 
\seqnum{A005709}, and \seqnum{A005710} in the OEIS.

\begin{lemma}\label{L:sumcoeff}
The sum of the coefficients of $f^{(t)}_n(x)$ is $f^{(t)}_n(1)=f^{(t)}_n$.
\end{lemma}
\begin{proof}
The sum of the coefficients of $f^{(t)}_n(x)$ can be expressed as
$f^{(t)}_n(1)$.  Putting $x=1$ into \eqref{e:f(x)} gives \eqref{e:f}
with $f^{(t)}_n$ replaced by $f^{(t)}_n(1)$.
\end{proof}

In the next lemma and theorem (which are generalizations of Lemma~13
and Theorem~14 in AE22) we employ the coefficient operator $[x^k]$
which denotes the coefficient of $x^k$ in the term it precedes.

\begin{lemma}\label{L:f=t}  
Let $f(t,n,k)=[x^k]f^{(t)}_n(x)$ and let $b(t,n,k)$ be the number of tilings of
an $n$-board with squares and $t$-ominoes that use exactly $k$
$t$-ominoes. Then $f(t,n,k)=b(t,n,k)$ for all $n$ and $k$.
\end{lemma}

\begin{proof}
This follows from Theorem~10 of AE22. The metatiles are the square and
$t$-omino. 
\end{proof}

\begin{theorem}\label{T:poly}
For $j\ge0$, $k\ge0$, $m\ge1$, $t\ge2$, and  $r=0,\ldots,m-1$,
\begin{equation}\label{e:poly} 
\chb{mj+r-(t-1)k}{k}{m,t}
=[x^k]\bigl(f^{(t)}_j(x)\bigr)^{m-r}\bigl(f^{(t)}_{j+1}(x)\bigr)^r.
\end{equation} 
\end{theorem}
\begin{proof}
From Identity~\ref{I:ch=chb},
$\tchb{mj+r-(t-1)k}{k}_{m,t}=\tch{mj+r}{k}_{m,t}$. From Lemma~\ref{L:bij},
$\tch{mj+r}{k}_{m,t}$ equals the number of ways to tile an ordered
$m$-tuple of $r$ $(j+1)$-boards followed by $m-r$ $j$-boards using $k$
$t$-ominoes (and $mj+r-kt$ squares). The number of such tilings of the
$m$-tuple of boards is
\[
\sum_{\substack{k_1\ge0,\,k_2\ge0,\,\ldots,\,k_m\ge0,\\k_1+k_2+\cdots+k_m=k}}
\Biggl(\prod_{i=1}^r b(t,j+1,k_i)\Biggr)
\Biggl(\prod_{i=r+1}^m b(t,j,k_i)\Biggr)
\]
in which the first product is omitted when $r=0$.
The coefficient of $x^k$ in 
$\bigl(f^{(t)}_{j+1}(x)\bigr)^r\bigl(f^{(t)}_j(x)\bigr)^{m-r}$ is
\begin{align*}  
&[x^k]
\Biggl(\prod_{i=1}^r \sum_{k_i=0}^{\floor{(j+1)/t}}f(t,j+1,k_i)x^{k_i}\Biggr)
\Biggl(\prod_{i=r+1}^m \sum_{k_i=0}^{\floor{j/t}}f(t,j,k_i)x^{k_i}\Biggr)\\
&\qquad=
[x^k]\sum_{k_1\ge0,k_2\ge0,\ldots,k_m\ge0}\Biggl(\prod_{i=1}^r f(t,j+1,k_i)\Biggr)
\Biggl(\prod_{i=r+1}^m f(t,j,k_i)\Biggr)x^{k_1+k_2+\cdots+k_m}\\
&\qquad=
\sum_{\substack{k_1\ge0,\,k_2\ge0,\,\ldots,\,k_m\ge0,\\k_1+k_2+\cdots+k_m=k}}
\Biggl(\prod_{i=1}^r f(t,j+1,k_i)\Biggr)
\Biggl(\prod_{i=1}^r f(t,j,k_i)\Biggr).
\end{align*}
The result then follows from Lemma~\ref{L:f=t}.  
\end{proof}

The following identity gives the sums of the $(1,t-1)$-antidiagonals
of the $\tchb{n}{k}_{m,t}$ triangle.  
\begin{identity}\label{I:adiagsum} 
For $t\ge2$, $j\ge0$, $m\ge1$, and $r=0,\ldots,m-1$,
\[
\sum_{k=0}^{\floor{(mj+r)/t}}\chb{mj+r-(t-1)k}{k}{m,t}
=\bigl(f^{(t)}_j\bigr)^{m-r}\bigl(f^{(t)}_{j+1}\bigr)^r.
\]
\end{identity}
\begin{proof}
Summing \eqref{e:poly} over all permitted $k$ gives the sum of all
coefficients of
\[
F(x)=\bigl(f^{(t)}_j(x)\bigr)^{m-r}\bigl(f^{(t)}_{j+1}(x)\bigr)^r
\]
which is $F(1)$ and equals
$\bigl(f^{(t)}_j\bigr)^{m-r}\bigl(f^{(t)}_{j+1}\bigr)^r$ by
Lemma~\ref{L:sumcoeff}.
\end{proof}

\section{Relation of the triangles to restricted combinations}\label{s:comb}

We now look at
$S^{(m,t)}(n,k)$, the number of subsets of $\Nset_n$ of
size $k$ such that the difference of any two elements of the subset
does not equal any element in the set
$\mathcal{Q}=\{m,2m,\ldots,(t-1)m\}$.  For example,
$S^{(2,3)}(5,0)=1$, $S^{(2,3)}(5,1)=5$, $S^{(2,3)}(5,2)=6$, and
$S^{(2,3)}(5,k>2)=0$ since the possible subsets of $\Nset_5$
such that no two elements in the subset differ by 2 or 4 are $\{\}$,
$\{1\}$, $\{2\}$, $\{3\}$, $\{4\}$, $\{5\}$, $\{1,2\}$, $\{2,3\}$,
$\{3,4\}$, $\{4,5\}$, $\{1,4\}$, and $\{2,5\}$.  There is a formula
for $S^{(m,t)}(n,k)$ in terms of sums of products of binomial
coefficients \cite{MS08}.  Here we will show that
$S^{(m,t)}(n,k)=\tchb{n+(t-1)(m-k)}{k}_{m,t}$ and hence obtain an
expression for the number of subsets in terms of coefficients of products
of $(1,t)$-bonacci polynomials which is a generalization of earlier
results \cite{KL91c,AE22}.  We first establish the following
bijection.

\begin{lemma}\label{L:ksub}
For $m,n\ge1$, $t\ge2$, and $k\ge0$, there is a bijection between the
$k$-subsets of $\Nset_n$ such that all pairs of elements taken from a
subset do not differ by an element from the set
$\mathcal{Q}=\{m,2m,\ldots,(t-1)m\}$, and the tilings of an
$(n+(t-1)m)$-board with $k$ $(1,m-1;t)$-combs and $n+(t-1)m-kt$
squares.
\end{lemma}
\begin{proof}
We label the cells of the $(n+(t-1)m)$-board from 1 to $n+(t-1)m$. If a
$k$-subset contains element $i$ then we place a comb so that its left
tooth occupies cell $i$. Notice that if $i=n$ then the rightmost tooth
occupies the final cell on the board. After placing combs
corresponding to each element of the subset, the rest of the board is
filled with squares of which there must be $n+(t-1)m-kt$. Conversely, the
tiling of any $(n+(t-1)m)$-board with $k$ combs corresponds to a
$k$-subset where no two elements differ by an element of $\mathcal{Q}$
 since the remaining teeth of
a comb whose leftmost tooth occupies cell $i$ lie on cells 
$i+m,i+2m,\ldots,i+(t-1)m$ which means none of these cells can be
occupied by the leftmost tooth of another comb.
\end{proof}

\begin{corollary}\label{C:S=chb}
For $m,n\ge1$, $t\ge2$, and $k\ge0$,
$S^{(m,t)}(n,k)=\tchb{n+(t-1)(m-k)}{k}_{m,t}$.
\end{corollary}
\begin{proof}
From Lemma~\ref{L:ksub}, $S^{(m,t)}(n,k)=\tch{n+(t-1)m}{k}_{m,t}$.
Identity~\ref{I:ch=chb} then gives the result.
\end{proof}

\begin{corollary}\label{C:A}
For $m,n\ge1$, $t\ge2$, the sum of the elements in the $n$-th
$(1,t-1)$-antidiagonal of 
$\tchb{n}{k}_{m,t}$ is the number of subsets of $\Nset_{n-(t-1)m}$
  chosen so that no two elements of the subsets differ by any member
  of the set $\{m,\ldots,(t-1)m\}$.
\end{corollary}
\begin{proof}
The elements in the $(1,t-1)$-antidiagonal are, for $k\ge0$, 
$\tchb{n-(t-1)k}{k}_{m,t}=S^{(m,t)}(n-(t-1)m,k)$ by Corollary~\ref{C:S=chb}.
Summing over all $k$ then gives the result.
\end{proof}

The next two corollaries follow from Theorem~\ref{T:poly} and
Identity~\ref{I:adiagsum}, respectively.

\begin{corollary}
For $j,k\ge0$, $m\ge1$, $t\ge2$, and $r=0,\ldots,m-1$, 
\[
S^{(m,t)}(mj+r,k)
=[x^k]\bigl(f^{(t)}_{j+t-1}(x)\bigr)^{m-r}\bigl(f^{(t)}_{j+t}(x)\bigr)^r.
\]
\end{corollary}

\begin{corollary}
For $j\ge0$, $m\ge1$, $t\ge2$, and $r=0,\ldots,m-1$, the number of
subsets of $\Nset_{mj+r}$ each of which lack pairs of elements that
differ by a multiple of $m$ up to $(t-1)m$ is
$\bigl(f^{(t)}_{j+t-1}\bigr)^{m-r}\bigl(f^{(t)}_{j+t}\bigr)^r$.
\end{corollary}

\section{Metatiles and digraphs}\label{s:meta}

The simplest metatiles when tiling with squares ($S$) and
$(1,m-1;t)$-combs ($C$) are the \textit{free square} ($S$), what we will refer
to as an \textit{$m$-comb} ($C^m$) which is $m$ interlocking combs
with no gaps, and the \textit{filled comb} ($CS^{(m-1)(t-1)}$) which
is a comb with all the gaps filled with squares. The $m=2$, $t=3$
instances of these are the first three metatiles depicted in
Fig.~\ref{f:meta}(a).

\begin{figure}
\begin{center}
\includegraphics[width=15cm]{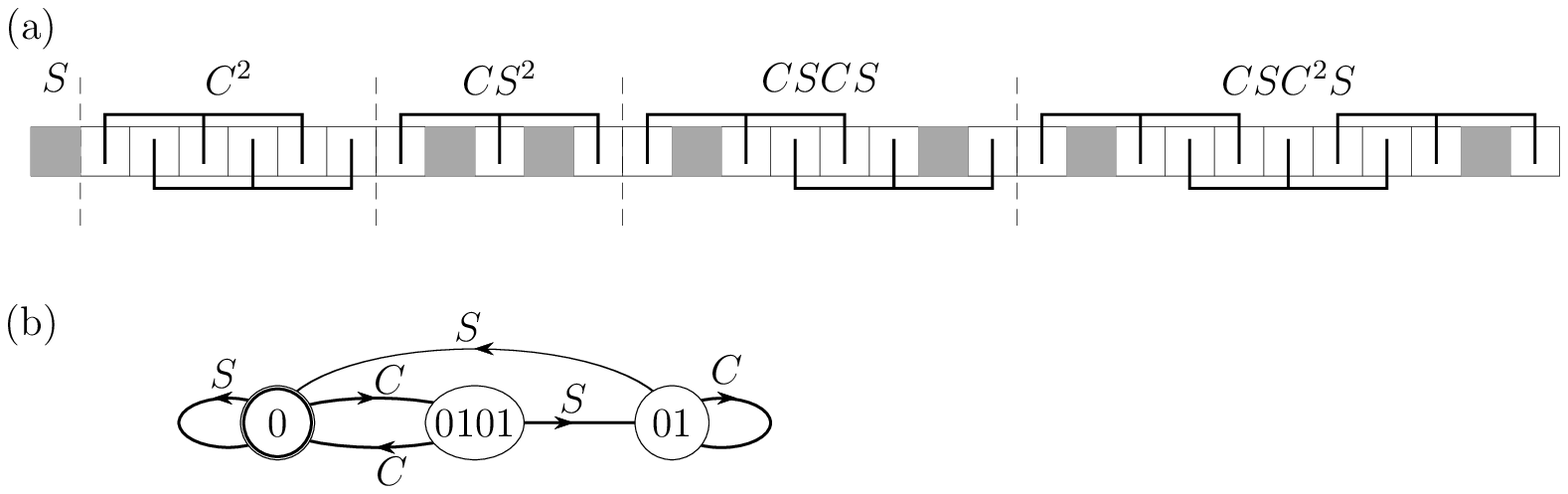}
\end{center}
\caption{Metatiles when tiling with squares and $(1,1;3)$-combs ($m=2$,
  $t=3$). (a)~A 31-board tiled with all the metatiles containing less
  than 6 tiles. Shaded (white) cells are occupied by squares
  (combs). Bold lines indicate which teeth belong to the same comb.
  Dashed lines show boundaries between metatiles. The symbolic
  representation is above each metatile. (b)~The digraph for
  generating metatiles.}
\label{f:meta}
\end{figure}

When $m=1$, the only metatiles are the two individual tiles
themselves: a square and a comb, which, as the gaps are of zero width,
is just a $t$-omino. When $m>1$, the only case when there is a finite
number of metatiles is when $t=2$ \cite{EA21}.  There are two cases
when there is a single infinite sequence of metatiles: the
$(m,t)=(3,2)$ case, which was dealt with in AE22, and when $m=2$ and
$t=3$. In the latter case, the metatiles are $S$, $C^2$, and $CSC^jS$
for $j\ge0$, as illustrated in Fig.~\ref{f:meta}(a).  This infinite
sequence of metatiles is analogous to that found for the $(m,t)=(3,2)$
case \cite[\S6]{AE22}: $CS$ has a single remaining unit-width
slot which can be filled either with an $S$, thus completing the metatile, or
with the left tooth of a $C$ (to give $CSC$) which again results in a
slot of unit width.

For a particular choice of types of tiles, a systematic way to
generate all metatiles and, in the simpler cases, obtain finite-order
recursion relations for the number of tilings is via a directed
pseudograph (henceforth referred to as a \textit{digraph}) in which
each arc represents the addition of a tile and each node represents
the current state of the yet-to-be-completed metatile
\cite{EA15,EA20}.  Any such digraph contains a \textit{0 node} which
represents the empty board or the completed metatile. The remaining
nodes are named using binary strings: the $i$-th digit of the string
is 0 (1) if the $i$-th cell, starting at the first unoccupied cell of
the incomplete metatile and ending at its last occupied cell, is empty
(filled). Thus all nodes (except the 0 node) start with 0 and end with
1. There is a bijection between each possible metatile and each path
on the digraph which starts and finishes at the 0 node without
visiting it in between. To obtain the symbolic representation of the
metatile, one simply reads off the names of the arcs along the path
and then simplifies the resulting expression by, for example,
replacing $CC$ by $C^2$. The digraph for generating metatiles when
tiling with squares and $(1,1;3)$-combs is shown in
Fig.~\ref{f:meta}(b).

\section{Further identities}\label{s:idn}

We start by deriving identities that apply to all the triangles and
later on obtain recursion relations for some particular instances of the
triangles after constructing the corresponding metatile-generating digraphs. 
The following three identities arise from considering the simplest
types of $n$-tile tilings.

\begin{identity}\label{I:chb=1}
For $n\ge0$, $m\ge1$, and $t\ge2$, $\tchb{n}{0}_{m,t}=1$.
\end{identity}
\begin{proof}
There is only one way to create an $n$-tile tiling without using any
combs, namely, the all-square tiling.
\end{proof}

\begin{identity}
For $n\ge0$, $m\ge1$, and $t\ge2$, 
$\tchb{n}{n}_{m,t}=\delta_{n\bmod m,0}$.
\end{identity}
\begin{proof}
The only way to tile without squares is the all $m$-comb tiling which
can only occur if the number of tiles is a multiple of $m$.
\end{proof}

\begin{identity}
For $n,m\ge1$ and $t\ge2$, 
\[
\chb{n}{1}{m,t}=
\begin{cases}0,&\text{if $n<(m-1)(t-1)+1$};
\\n-(m-1)(t-1),&\text{otherwise}.\end{cases}
\]
\end{identity}
\begin{proof}
Any $n$-tile tiling using exactly one $(1,m-1;t)$-comb must have a
filled comb which itself contains $(m-1)(t-1)+1$ tiles. Thus there can
be no $n$-tile tilings using 1 comb that use less than this number of
tiles. If $n\ge (m-1)(t-1)+1$, the tiling consists of a filled comb
and $n-(m-1)(t-1)-1$ free squares which gives a total of
$n-(m-1)(t-1)$ metatile positions in which the filled comb can be
placed.
\end{proof}

The pattern of zeros seen in the triangles is a result of the
following identity.
\begin{identity}
For $j\ge1$, $m,t\ge2$, $p=1,\ldots,m-1$, and $r=1-(t-2)p,\ldots,p$, 
\[
\chb{mj-r}{mj-p}{m,t}=0.
\]
\end{identity}
\begin{proof}
We first derive an expression for $K$, the maximum number of combs
that can be used in the tiling of an $(mJ+R)$-board where
$R=0,\ldots,m-1$.  From Lemma~\ref{L:bij}, $K$ is also the maximum
number of $t$-ominoes that can be used in the tiling of $R$
$(J+1)$-boards and $m-R$ $J$-boards. Then it is straightforward to show that
\begin{equation}\label{e:K} 
K=\begin{cases}\displaystyle\frac{m\bigl(J-(J\bmod t)\bigr)}{t}, 
& \text{if $J\bmod t<t-1$};\\
 \displaystyle\frac{m(J-t+1)}{t}+R,
& \text{if $J\bmod t=t-1$}.\end{cases}
\end{equation} 
From Identity~\ref{I:ch=chb}, 
\[
\chb{mj-r}{mj-p}{m,t}=\ch{tmj-r-(t-1)p}{mj-p}_{m,t}.
\]
Writing $tmj-r-(t-1)p$ in the form $mJ+R$, if $J=tj-s$ where
$s=1,\ldots,t$ then $R=sm-r-(t-1)p$. The condition that $0\leq R<m$
gives $(s-1)m<r+(t-1)p\leq sm$. This condition is compatible with the
minimum and maximum values $r+(t-1)p$ can take which are,
respectively, 2 and $tm$.
From \eqref{e:K} we find for $s>1$
that $K=m(j-1)$ which is always less than $mj-p$. When $s=1$,
$K=mj-p-r-(t-2)p$. Since $r+(t-2)p\geq1$ we have $K<mj-p$ in this case
as well.
\end{proof}

The following identity explains the entries that appear at the
vertical boundaries of the nonzero parts of the triangles and start
with rising powers of ascending positive integers.
Identities~\ref{I:j^p} and \ref{I:binom^p} reduce to Identity~24 of
AE22 when $t=2$.
\begin{identity}\label{I:j^p}
For $j,m\ge1$, $t\ge2$, $s=0,\ldots,t-2$, and $r=0,\ldots,m$, 
\begin{align*} 
\chb{m(j+s-1)+r}{m(j-1)}{m,t}&=\binom{j+s-1}{s}^{m-r}\binom{j+s}{s+1}^r\\
&=\begin{cases}
j^r, & \text{if $s=0$};\\
\biggl(\dfrac{j(j+1)\cdots(j+s-1)}{s!}\biggr)^m
\biggl(\dfrac{j+s}{s+1}\biggr)^r,& \text{if $s>0$}.
\end{cases}
\end{align*} 
\end{identity}
\begin{proof}
From Identity~\ref{I:ch=chb},
\[ 
\chb{m(j+s-1)+r}{m(j-1)}{m,t}=\ch{m(t(j-1)+s)+r}{m(j-1)}_{m,t}.
\]
By Lemma~\ref{L:bij}, this is the number of ways to tile $m-r$ boards
of length $t(j-1)+s$ and $r$ boards of length $t(j-1)+s+1$ with $m(j-1)$
$t$-ominoes (and $sm+r$ squares). As $s+1<t$, each of the $m$ boards
always contains exactly $j-1$ $t$-ominoes. A board of
length $t(j-1)+s$ has $s$ squares and so there are $j+s-1$ metatile
positions in which to put the squares (the rest being filled by
$t$-ominoes) and thus $\tbinom{j+s-1}{s}$ ways to tile it. Likewise,
a board of length $t(j-1)+s+1$ has $j+s$ metatile positions and so
there are $\tbinom{j+s}{s+1}$ ways to tile it. The result follows from
the numbers of each type of board.
\end{proof}

The next identity explains the rising
powers of integers on non-vertical rays of entries at the
boundaries of the nonzero parts of the triangles.
\begin{identity}\label{I:binom^p}
For $m,j\ge1$, $t\ge2$, and $p=0,\ldots,m$, 
\[
\chb{mj+(t-2)p}{mj-p}{m,t}=\binom{j+t-2}{t-1}^p.
\]
\end{identity}
\begin{proof}
 From Identity~\ref{I:ch=chb} we have
\[
\chb{mj+(t-2)p}{mj-p}{m,t}
=\ch{tmj-p}{mj-p}_{m,t}=\ch{m(jt-1)+m-p}{(m-p)j+p(j-1)}_{m,t},
\]
which is also the number of ways to tile $m-p$ $jt$-boards and $p$
boards of length $jt-1$ using $(m-p)j+p(j-1)$ $t$-ominoes and $(t-1)p$
squares. The $jt$-boards are completely filled by $j$ $t$-ominoes and
the $p$ $(jt-1)$-boards each have $j-1$ $t$-ominoes and $t-1$
squares. As on these $p$ shorter boards there are $j+t-2$ tiles in total,
there are $\tbinom{j+t-2}{t-1}$ ways to tile each of them which leads
to a total of $\tbinom{j+t-2}{t-1}^p$ tilings for the set of boards.
\end{proof}

The following two identities are generalizations of Identities~25 and
26 in AE22.
\begin{identity}
For $j,m\ge1$ and $t\ge2$, 
\[
\chb{mj+t-1}{mj-1}{m,t}=m\binom{j+t-1}{t}.
\]
\end{identity}
\begin{proof}
From Identity~\ref{I:ch=chb}, $\tchb{mj+t-1}{mj-1}_{m,t}=\tch{tmj}{mj-1}_{m,t}$,
which, from Lemma~\ref{L:bij}, is the number of ways to tile an
$m$-tuple of $jt$-boards with $mj-1$ $t$-ominoes and $t$ squares. As the
length of each board is a multiple of $t$, all the squares must lie on the same
board. On such a board there are $j-1$ $t$-ominoes and $t$ squares
making $j+t-1$ tiles in total. Hence 
there are $\tbinom{j+t-1}{t}$ possible ways to tile it. 
As there are $m$ possible boards on which to place all the squares, the
result follows.
\end{proof}

\begin{identity}
For $t\ge2$ and $m,j\ge1$ provided $mj\geq2$, 
\[
\chb{mj+2(t-1)}{mj-2}{m,t}=\begin{cases}
\displaystyle\binom{m}{2}, & \text{if $j=1, m>1$};\\
\displaystyle m\binom{j+2(t-1)}{2t}+\binom{m}{2}\binom{j+t-1}{t}^2, 
& \text{if $m,j>1$},\\
\displaystyle \binom{j+2(t-1)}{2t}, 
& \text{if $m=1, j>1$}.\\
\end{cases}
\]
\end{identity}
\begin{proof}
From Identity~\ref{I:ch=chb},
$\tchb{mj+2(t-1)}{mj-2}_{m,t}=\tch{tmj}{mj-2}_{m,t}$, which, from
Lemma~\ref{L:bij}, is the number of ways to tile an $m$-tuple of
$jt$-boards with $mj-2$ $t$-ominoes and $2t$ squares.  If $j>1$, all
$2t$ squares can be on the same $jt$-board which, with the $j-2$
$t$-ominoes on that board, makes $j-2+2t$ tiles in total and hence
$\tbinom{j+2(t-1)}{2t}$ tilings of it. With $m$ boards to choose from,
this gives the first term on the right-hand sides of the identity when
$j>1$. Otherwise, if $m>1$, two of the boards have $t$ squares each. There are
$\tbinom{j+t-1}{t}$ ways to tile each of those boards and $\tbinom{m}{2}$
ways to choose them.
\end{proof}

The following identity is a generalization of the previous two. 
\begin{identity}
For $s\ge1$, $t\ge2$, and $m,j\ge1$ provided $mj\ge s$, 
\[
\chb{mj+s(t-1)}{mj-s}{m,t}=
\sum_{\substack{r_i\ge1;\\r_1+\cdots+r_p=s}}
\!\!\!\binom{m}{p}\prod_{i=1}^p\binom{j+r_i(t-1)}{r_it},
\]
where the sum is over compositions of $s$, $p$ is the number of parts
of the composition, and $\tbinom{a}{b}$ is understood to equal zero if
$a<b$.
\end{identity}
\begin{proof}
From Identity~\ref{I:ch=chb},
$\tchb{mj+s(t-1)}{mj-s}_{m,t}=\tch{tmj}{mj-s}_{m,t}$, which, from
Lemma~\ref{L:bij}, is the number of ways to tile an $m$-tuple of
$jt$-boards with $mj-s$ $t$-ominoes and $st$ squares. We partition the
squares into $p$ parts of sizes $r_it$ where $r_i\in\Zset^+$ 
such that $r_1+\cdots+r_p=s$.
A $jt$-board containing $r_it$ squares has $j-r_i$ $t$-ominoes and
thus $j-r_i+r_it$ tiles in total and so $\tbinom{j+r_i(t-1)}{r_it}$
possible ways to tile it. There are $\binom{m}{p}$ ways to choose
which of the $m$ boards have any squares.
\end{proof}

In order to truly deserve to be called a Pascal-like triangle, a
triangle ought to have a portion where Pascal's recurrence is
obeyed. We now show that this is the case for our triangles by using a
result from a study on restricted combinations \cite{MS08} to extend
and prove Conjecture~30 of AE22.
\begin{theorem}
For integers $k\ge0$, $m\geq1$, $t\geq2$, and $n>(m-1)(t-1)k$, 
\begin{equation}\label{e:PRchb} 
\chb{n}{k}{m,t}=\chb{n-1}{k}{m,t}+\chb{n-1}{k-1}{m,t}.
\end{equation} 
\end{theorem}
\begin{proof}
The result holds for $k=0$ since $\tchb{n\ge0}{1}_{m,t}=1$ by
Identity~\ref{I:chb=1} and $\tchb{n<0}{k}_{m,t}=\tchb{n}{k<0}_{m,t}=0$
by definition.  Mansour and Sun give the result in Theorem~3.5 of
their paper~\cite{MS08}, when rewritten in our own notation, that for
any integers $m,k\geq1$, and $t\geq2$,
\begin{equation}\label{e:Srr} 
S^{(m,t)}(N,k)=S^{(m,t)}(N-1,k)+S^{(m,t)}(N-t,k-1),
\end{equation} 
provided that $N\geq m(t-1)(k-1)$. However, the condition for this relation
between numbers of subsets to hold should read $N>m(t-1)(k-1)$
(personal communication with Mansour). By Corollary~\ref{C:S=chb},
$n=N+(t-1)(m-k)$ and we can rewrite \eqref{e:Srr} as
\[
\chb{N+(t-1)(m-k)}{k}{m,t}=
\chb{N+(t-1)(m-k)-1}{k}{m,t}+
\chb{N+(t-1)(m-k+1)-t}{k-1}{m,t}
\]
which reduces to \eqref{e:PRchb}. The corrected condition becomes
$n-(t-1)(m-k)>m(t-1)(k-1)$ which gives the condition in our theorem.
\end{proof}

We now turn to obtaining recursion relations for particular instances
of the triangles. For all but the last triangle we consider, we
require the following theorem which extends a result proved elsewhere
for tilings of an $n$-board when the digraph has a common node
\cite[Theorem~5.4 and Identity~5.5]{EA15} to also include $n$-tile
tilings of boards.

\begin{theorem}\label{T:CN}
For a digraph possessing a common node, let $l_{\mathrm{o}i}$ be the
length of the $i$-th outer cycle ($i=1,\ldots,N\rb{o}$), let $L_r$ be
the length of the $r$-th inner cycle ($r=1,\ldots,N$) and let $K_r$ be the
number of combs it contains, and let $l_{\mathrm{c}i}$ be the length of
the $i$-th common circuit ($i=1,\ldots,N\rb{c}$) and
let $k_{\mathrm{c}i}$ be the number of combs it contains. Then for all
integers $n$ and $k$, 
\begin{align} 
\label{e:CNBn}
B_n&=\delta_{n,0}+ 
\sum_{r=1}^N (B_{n-L_r}-\delta_{n,L_r})+
\sum_{i=1}^{N\rb{o}}
\biggl(B_{n-l_{\mathrm{o}i}}-\sum_{r=1}^NB_{n-l_{\mathrm{o}i}-L_r}\biggr)
+\sum_{i=1}^{N\rb{c}} B_{n-l_{\mathrm{c}i}},\\
B_{n,k}&=\delta_{n,0}\delta_{k,0}+ 
\sum_{r=1}^N (B_{n-L_r,k-K_r}-\delta_{n,L_r}\delta_{k,K_r})+
\sum_{i=1}^{N\rb{o}}
\biggl(B_{n-l_{\mathrm{o}i},k-k_{\mathrm{o}i}}
-\sum_{r=1}^NB_{n-l_{\mathrm{o}i}-L_r,k-k_{\mathrm{o}i}-K_r}\biggr)\nonumber\\
&\qquad\mbox{}+\sum_{i=1}^{N\rb{c}} B_{n-l_{\mathrm{c}i},k-k_{\mathrm{c}i}},
\label{e:CNBnk}
\end{align} 
where $B_{n<0}=B_{n,k<0}=B_{n<k,k}=0$. If the lengths of the cycles
and circuits are calculated as the number of tiles (the total contribution
made to the number of cells occupied) then
$B_n$ is the number of $n$-tile tilings (the number of tilings of an
$n$-board) and $B_{n,k}$ is the number of such tilings that use $k$ combs.
\end{theorem}

In the proofs of Identities \ref{I:rr23}, \ref{I:B23}, \ref{I:rr24},
\ref{I:B24}, \ref{I:rr42}, and \ref{I:B42} which use
Theorem~\ref{T:CN} (and Identities \ref{I:rr25} and \ref{I:B25} which
use Theorem~\ref{T:PCN}), the lengths of the cycles and circuits are
the number of tiles they contain. In the proofs of Identities
\ref{I:A23}, \ref{I:A24}, \ref{I:A42}, and \ref{I:A25}, the lengths of
the cycles and circuits are the total number of cells that the tiles
along the arcs occupy. An $S$ occupies 1 cell whereas a $C$ occupies
$t$ cells. Thus if $L$ is the length of a cycle or circuit containing
$K$ combs when finding the recursion relations for $n$-tile tilings
then $L'=L+(t-1)K$ is the length of that cycle or circuit when the
recursion relations are for the tilings of an $n$-board.

\begin{identity}\label{I:rr23}
For all $n,k\in\mathbb{Z}$,
\begin{multline}\label{e:rr23}
\chb{n}{k}{2,3}=\delta_{n,0}\delta_{k,0}-\delta_{n,1}\delta_{k,1}
+\chb{n-1}{k}{2,3}+\chb{n-1}{k-1}{2,3}
-\chb{n-2}{k-1}{2,3}+\chb{n-2}{k-2}{2,3}\\
+\chb{n-3}{k-1}{2,3}-\chb{n-3}{k-3}{2,3}.
\end{multline} 
\end{identity}
\begin{proof}
The digraph for tiling with squares and $(1,1;3)$-combs has a single
inner cycle connecting the 01 node to itself by a $C$
(Fig.~\ref{f:meta}(b)). Hence 01 is the common node and
$L_1=K_1=1$. There are 2 outer cycles ($S$ and $C^2$) and so
$l_{\mathrm{o}1}=1$, $k_{\mathrm{o}1}=0$, and
$l_{\mathrm{o}2}=k_{\mathrm{o}2}=2$. There is a single common circuit
 ($CS^2$) which gives $l_{\mathrm{c}1}=3$ and $k_{\mathrm{c}1}=1$.
\end{proof}

\begin{identity}\label{I:B23}
If $B_n$ is the sum of the $n$-th row of $\tchb{n}{k}_{2,3}$ then for
all $n$,
\[
B_n=\delta_{n,0}-\delta_{n,1}-\delta_{n,2}+2B_{n-1},
\]
where $B_{n<0}=0$.
\end{identity}
\begin{proof}
Sum each term in \eqref{e:rr23} over all $k$ or use \eqref{e:PCNBn}.
\end{proof}
As defined above,
$(B_n)_{n\ge0}=1,1,2,4,8,16,32,64,128,256,\ldots$ is \seqnum{A011782}.

\begin{identity}\label{I:A23}
If $A_n$ is the sum of the $n$-th $(1,2)$-antidiagonal of
$\tchb{n}{k}_{2,3}$ then for all $n$,
\[
A_n=\delta_{n,0}-\delta_{n,3}+A_{n-1}+A_{n-3}-A_{n-4}+A_{n-5}+A_{n-6}-A_{n-9},
\]
where $A_{n<0}=0$.
\end{identity}
\begin{proof}
By Identity~\ref{I:ch=chb}, the $n$-th $(1,2)$-antidiagonal of
$\tchb{n}{k}_{2,3}$ is the $n$-th row of $\tch{n}{k}_{2,3}$.  
From the definition of the latter triangle, $A_n$ is the number of
tilings of an $n$-board using squares and $(1,1;3)$-combs
and is given by \eqref{e:PCNBn} (with $B_n$ replaced by $A_n$) 
applied to the same digraph as in the proof of Identity~\ref{I:rr23}
but with the following changes made to the lengths: 
$L_1=3$, $l_{\mathrm{o}2}=6$, $l_{\mathrm{c}1}=5$.
\end{proof}
As defined above,
$(A_n)_{n\ge0}=1,1,1,1,1,2,4,6,9,12,16,24,36,54,81,117,\ldots$ 
is \seqnum{A224809}. From
Corollary~\ref{C:A}, $A_n$ is the number of subsets of
$\Nset_{n-4}$ chosen so that no two elements differ by 2 or 4.

\begin{figure}[!b]
\begin{center}
\includegraphics[width=10cm]{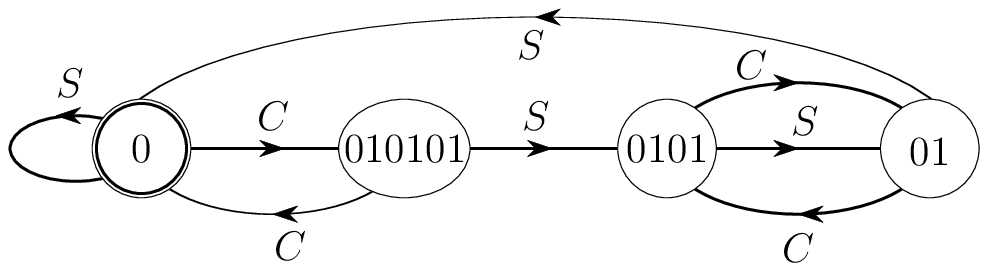}
\end{center}
\caption{Digraph for generating metatiles when tiling with squares and
  $(1,1;4)$-combs ($m=2$, $t=4$).}
\label{f:24}
\end{figure}

\begin{identity}\label{I:rr24}
For all $n,k\in\mathbb{Z}$,
\begin{multline}\label{e:rr24}
\chb{n}{k}{2,4}=\delta_{n,0}\delta_{k,0}-\delta_{n,2}(\delta_{k,1}
+\delta_{k,2})
+\chb{n-1}{k}{2,4}+\chb{n-2}{k-1}{2,4}+2\chb{n-2}{k-2}{2,4}\\
-\chb{n-3}{k-1}{2,4}-\chb{n-3}{k-2}{2,4}
+\chb{n-4}{k-1}{2,4}+\chb{n-4}{k-2}{2,4}
-\chb{n-4}{k-3}{2,4}-\chb{n-4}{k-4}{2,4}.
\end{multline} 
\end{identity}
\begin{proof}
The digraph for tiling with squares and $(1,1;4)$-combs has 2 inner
cycles ($SC$ and $C^2$) both of which pass though the 0101 and 01
nodes (Fig.~\ref{f:24}). We choose 0101 as the common node.  We see that $L_1=L_2=2$,
$K_1=1$, and $K_2=2$. There are 2 outer cycles ($S$ and $C^2$) and so
$l_{\mathrm{o}1}=1$, $k_{\mathrm{o}1}=0$, and
$l_{\mathrm{o}2}=k_{\mathrm{o}2}=2$. There are 2 common circuits:
$CS\{S,C\}S$ where $X\{Y,Z\}$ means $XY$ and $XZ$. Hence
$l_{\mathrm{c}1}=l_{\mathrm{c}2}=4$, $k_{\mathrm{c}1}=1$, and
$k_{\mathrm{c}2}=2$.
\end{proof}

\begin{identity}\label{I:B24}
If $B_n$ is the sum of the $n$-th row of $\tchb{n}{k}_{2,4}$ then for
all $n$,
\[
B_n=\delta_{n,0}-2\delta_{n,2}+B_{n-1}+3B_{n-2}-2B_{n-3},
\]
where $B_{n<0}=0$.
\end{identity}
\begin{proof}
Sum each term in \eqref{e:rr24} over all $k$ or use \eqref{e:PCNBn}.
\end{proof}
As defined above,
$(B_n)_{n\ge0}=1,1,2,3,7,12,27,49,106,199,419,\ldots$ is \seqnum{A099163}.

The proofs of the following identity and Identities~\ref{I:A42} and
\ref{I:A25} are analogous to that of Identity~\ref{I:A23}. We just need
to find the modified lengths of the cycles and circuits in the digraph
before using the theorem giving the recursion relation.
 
\begin{identity}\label{I:A24}
If $A_n$ is the sum of the $n$-th $(1,3)$-antidiagonal of
$\tchb{n}{k}_{2,4}$ then for all $n$,
\[
A_n=\delta_{n,0}-\delta_{n,5}-\delta_{n,8}+
A_{n-1}+A_{n-5}-A_{n-6}+A_{n-7}+2A_{n-8}-A_{n-9}+A_{n-10}-A_{n-13}-A_{n-16},
\]
where $A_{n<0}=0$.
\end{identity}
\begin{proof}
We use the same digraph and associated parameters
as in the proof of Identity~\ref{I:rr24}
except that
$L_1=5$, $L_2=8$, $l_{\mathrm{o}2}=8$, $l_{\mathrm{c}1}=7$, 
and $l_{\mathrm{c}2}=10$.
\end{proof}
As defined above,
$(A_n)_{n\ge0}=1,1,1,1,1,1,1,2,4,6,9,12,16,20,25,35,\ldots$
is \seqnum{A224808}. From
Corollary~\ref{C:A}, $A_n$ is the number of subsets of
$\Nset_{n-6}$ chosen so that no two elements differ by 2, 4, or 6.

\begin{identity}\label{I:rr42}
For all $n,k\in\mathbb{Z}$,
\zmlg
\begin{multline}\label{e:rr42}
\chb{n}{k}{4,2}=\delta_{n,0}\delta_{k,0}-\delta_{n,2}\delta_{k,1}
-\delta_{n,3}\delta_{k,2}-\delta_{n,4}\delta_{k,4}
+\chb{n-1}{k}{4,2}+\chb{n-2}{k-1}{4,2}
-\chb{n-3}{k-1}{4,2}+\chb{n-3}{k-2}{4,2}\\
+\chb{n-4}{k-1}{4,2}+\chb{n-4}{k-3}{4,2}+2\chb{n-4}{k-4}{4,2}
+\chb{n-5}{k-2}{4,2}+2\chb{n-5}{k-3}{4,2}-\chb{n-5}{k-4}{4,2}\\
-\chb{n-6}{k-3}{4,2}-\chb{n-6}{k-5}{4,2}
-\chb{n-7}{k-4}{4,2}-\chb{n-7}{k-5}{4,2}-\chb{n-7}{k-6}{4,2}
-\chb{n-8}{k-7}{4,2}-\chb{n-8}{k-8}{4,2}.
\end{multline} 
\rmlg
\end{identity}
\begin{proof}
The digraph for tiling with squares and $(1,3;2)$-combs (which are
also called $(1,3)$-fences) has 3 inner cycles all of which contain
the nodes 001 and 01 (Fig.~\ref{f:42}). We choose 001 as the common
node. The cycles, given as lists of arcs starting from 001, are
$\{S,C\{S,C^2\}\}C$. Hence $L_i=2,3,4$ and $K_i=1,2,4$, respectively,
for $i=1,2,3$. There are 5 outer cycles:
$S,C^2\{S\{S,CS\},C\{S,C\}\}$. Thus $l_{\mathrm{o}i}=1,4,5,4,4$ and
$k_{\mathrm{o}i}=0,2,3,3,4$, respectively, for $i=1,\ldots,5$. There
are 8 common circuits: $C\{S,CSC^2\}\{S^2,C\{S^2,C\{S,CS\}\}\}$. Hence
$l_{\mathrm{c}i}=4,5,5,6,7,8,8,9$ and
$k_{\mathrm{c}i}=1,2,3,4,4,5,6,7$, respectively, for $i=1,\ldots,8$.
\end{proof}

\begin{figure}
\begin{center}
\includegraphics[width=8cm]{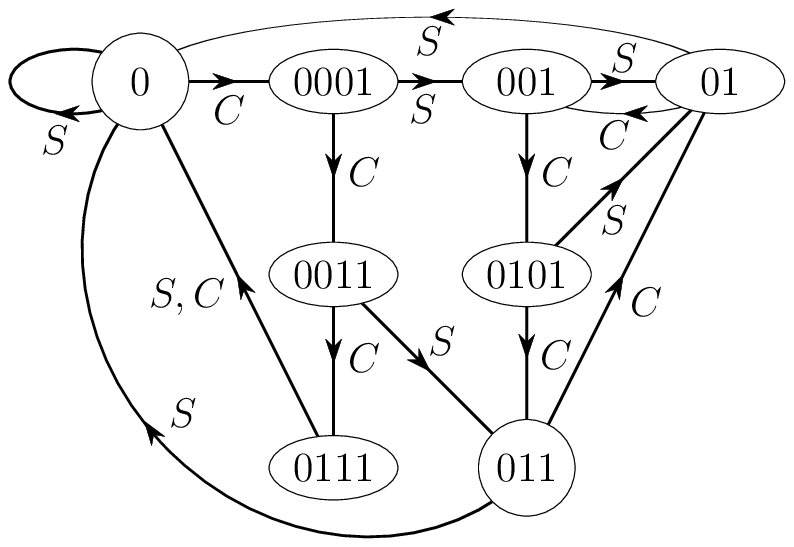}
\end{center}
\caption{Digraph for generating metatiles when tiling with squares and
  $(1,3;2)$-combs ($m=4$, $t=2$).}
\label{f:42}
\end{figure}

\begin{identity}\label{I:B42}
If $B_n$ is the sum of the $n$-th row of $\tchb{n}{k}_{4,2}$ then for
all $n$,
\[
B_n=\delta_{n,0}-\delta_{n,2}-\delta_{n,3}-\delta_{n,4}
+B_{n-1}+B_{n-2}+4B_{n-4}+2B_{n-5}-2B_{n-6}-3B_{n-7}-2B_{n-8},
\]
where $B_{n<0}=0$.
\end{identity}
\begin{proof}
Sum each term in \eqref{e:rr42} over all $k$ or use \eqref{e:PCNBn}.
\end{proof}
As defined above,
$(B_n)_{n\ge0}=1,1,1,1,5,12,21,34,70,155,318,610,\ldots$
has the generating function 
$(1-x-x^3)/((1-2x)(1-x^2)(1+2x^2+x^3+x^4))$.

\begin{identity}\label{I:A42}
If $A_n$ is the sum of the $n$-th antidiagonal of
$\tchb{n}{k}_{4,2}$ then for all $n$,
\begin{multline*} 
A_n=\delta_{n,0}-\delta_{n,3}-\delta_{n,5}-\delta_{n,8}+
A_{n-1}+A_{n-3}-A_{n-4}+2A_{n-5}+2A_{n-7}+4A_{n-8}-2A_{n-9}\\
-2A_{n-11}-A_{n-12}-A_{n-13}-A_{n-15}-A_{n-16},
\end{multline*} 
where $A_{n<0}=0$.
\end{identity}
\begin{proof}
We use the same digraph and associated parameters
as in the proof of Identity~\ref{I:rr42}
except that
$L_1=3$, $L_2=5$, $L_3=8$, $l_{\mathrm{o}i}=6,8,7,8$ for
$i=2,\ldots,5$,
and for $i=1,\ldots,8$, $l_{\mathrm{c}i}=5,7,8,10,11,13,14,16$.
\end{proof}
As defined above,
$(A_n)_{n\ge0}=1,1,1,1,1,2,4,8,16,24,36,54,81,135,225,\ldots$ (after
removing the first four 1s) is
\seqnum{A031923}. From Corollary~\ref{C:A}, $A_n$ is the number of
subsets of $\Nset_{n-4}$ chosen so that no two elements differ by
4.

\begin{figure}
\begin{center}
\includegraphics[width=13cm]{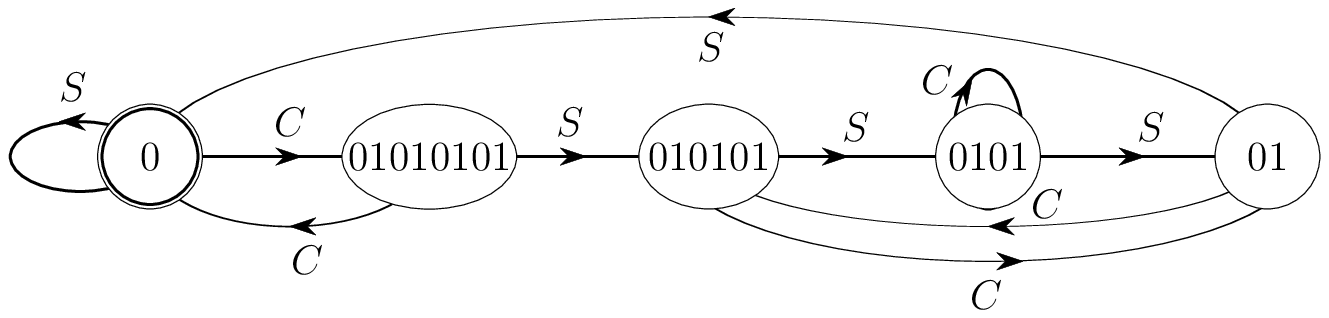}
\end{center}
\caption{Digraph for generating metatiles when tiling with squares and
  $(1,1;5)$-combs ($m=2$, $t=5$).}
\label{f:25}
\end{figure}

\begin{identity}\label{I:rr25}
For all $n,k\in\mathbb{Z}$,
\zmlg
\begin{multline}\label{e:rr25}
\chb{n}{k}{2,5}=\delta_{n,0}\delta_{k,0}-\delta_{n,1}\delta_{k,1}
-\delta_{n,2}\delta_{k,2}+\delta_{n,3}(\delta_{k,3}-\delta_{k,1})
+\chb{n-1}{k}{2,5}+\chb{n-1}{k-1}{2,5}-\chb{n-2}{k-1}{2,5}\\
+2\chb{n-2}{k-2}{2,5}
+\chb{n-3}{k-1}{2,5}-\chb{n-3}{k-2}{2,5}-2\chb{n-3}{k-3}{2,5}
-\chb{n-4}{k-1}{2,5}\\
+\chb{n-4}{k-2}{2,5}+\chb{n-4}{k-3}{2,5}-\chb{n-4}{k-4}{2,5}+\chb{n-5}{k-1}{2,5}
-2\chb{n-5}{k-3}{2,5}+\chb{n-5}{k-5}{2,5}.
\end{multline} 
\rmlg
\end{identity}
\begin{proof}
The digraph for tiling with squares and $(1,1;5)$-combs has 3 inner
cycles but no common node (Fig.~\ref{f:25}). If the loop at the 0101
node were not present, the digraph would have a common node. Using
terminology and notation we introduce in the appendix, the loop at
0101 is an errant loop and has length $L_0=1$ and number of combs
$K_0=1$. We take the 010101 node as the pseudo-common node (we could
have also chosen the 01 node instead). There are two common circuits, $CSCS$
and $CS^4$, the first of which is plain. Thus
$l_{\mathrm{c}1}=l_{\mathrm{pc}1}=4$,
$k_{\mathrm{c}1}=k_{\mathrm{pc}1}=2$, $l_{\mathrm{c}2}=5$,
$k_{\mathrm{c}2}=1$, $N_{\mathrm{c}}=2$, and $N_{\mathrm{pc}}=1$.  Of
the other two inner cycles, $C^2$ is plain, $S^2C$ is not. Thus
$L_1=2$, $K_1=2$, $L_2=3$, and $K_2=1$. The outer cycles are
$S$ and $C^2$ and are both plain. Hence $l_{\mathrm{o}1}=1$,
$k_{\mathrm{o}1}=0$, $l_{\mathrm{o}2}=k_{\mathrm{o}2}=2$, and
$N_{\mathrm{o}}=2$. The identity then follows from applying \eqref{e:PCNBnk}.
\end{proof}

\begin{identity}\label{I:B25}
If $B_n$ is the sum of the $n$-th row of $\tchb{n}{k}_{2,5}$ then for
all $n$,
\[
B_n=\delta_{n,0}-\delta_{n,1}-\delta_{n,2}+2B_{n-1}+B_{n-2}-2B_{n-3},
\]
where $B_{n<0}=0$.
\end{identity}
\begin{proof}
Sum each term in \eqref{e:rr25} over all $k$ or use \eqref{e:PCNBn}.
\end{proof}
As defined above,
$(B_n)_{n\ge0}=1,1,2,3,6,11,22,43,86,171,342,683,\ldots$ 
is \seqnum{A005578}.

\begin{identity}\label{I:A25}
If $A_n$ is the sum of the $n$-th $(1,4)$-antidiagonal of
$\tchb{n}{k}_{2,5}$ then for all $n$,
\begin{multline*} 
A_n=\delta_{n,0}-\delta_{n,5}-\delta_{n,7}-\delta_{n,10}+\delta_{n,15}+
A_{n-1}+A_{n-5}-A_{n-6}+A_{n-7}-A_{n-8}+A_{n-9}\\+2A_{n-10}
-A_{n-11}+A_{n-12}-2A_{n-15}+A_{n-16}-2A_{n-17}-A_{n-20}+A_{n-25}
\end{multline*} 
where $A_{n<0}=0$.
\end{identity}
\begin{proof}
We use the same digraph and associated parameters
as in the proof of Identity~\ref{I:rr25}
except that
$L_0=5$, $L_1=10$, $L_2=7$, $l_{\mathrm{o}2}=10$,
$l_{\mathrm{c}1}=l_{\mathrm{pc}1}=12$, and $l_{\mathrm{pc}2}=9$.
\end{proof}
As defined above,
$(A_n)_{n\ge0}=1,1,1,1,1,1,1,1,1,2,4,6,9,12,16,20,25,30,36,48,64,\ldots$ 
 is
\seqnum{A224811}. From Corollary~\ref{C:A}, $A_n$ is the number of
subsets of $\Nset_{n-8}$ chosen so that no two elements differ by
2, 4, 6, or 8.

\section{Discussion}
In this paper and AE22 we considered tiling-derived triangles whose
entries were shown to be numbers of $k$-subsets of $\Nset_n$ such that
no two elements of the subset differ by an element in a set
$\mathcal{Q}$ of disallowed differences.  In AE22, $\mathcal{Q}=\{m\}$
for fixed $m\in\Zset^+$, whereas in the present paper,
$\mathcal{Q}=\{m,2m,\ldots,(t-1)m\}$, where $t=2,3,\ldots$.  One is
then led to ask whether there is a correspondence between restricted
combinations specified by other types of $\mathcal{Q}$ and
tilings. When $\mathcal{Q}=\Nset_q$ for some $q\in\Zset^+$, using the
same ideas as in the proof of Lemma~\ref{L:ksub}, it is
straightforward to show that there is a bijection between the tilings
of an $(n+q)$-board using $k$ $(q+1)$-ominoes and squares and the
number of $k$-subsets. However, the corresponding $n$-tile tilings
triangles are just Pascal's triangle for any $q$.  In order to obtain
a tiling interpretation of restricted combinations with other classes
of $\mathcal{Q}$ one needs a form of a tiling where some parts of the
tiles are allowed to overlap with parts of other tiles. This will be
explored in depth in another paper.  Whether or not such tiling
schemes can be used to generate further aesthetically pleasing
families of number triangles remains to be seen.

From some of the entries in the OEIS that give the sums of the
$(1,t-1)$-antidiagonals of the triangle (see \seqnum{A224809},
\seqnum{A224808}, and \seqnum{A224811}) it appears that the number of
subsets of $\Nset_{n-(t-1)m}$ whose elements do not differ by an
element of the set $\{m,2m,\ldots,(t-1)m\}$ is also the number of
permutations $\pi$ of $\Nset_n$ such that $\pi(i)-i\in\{-m,0,(t-1)m\}$
for all $i\in\Nset_n$.  This is indeed true in general as we will
demonstrate combinatorially using combs and fences elsewhere.

In AE22 it was noted that the $\tchb{n}{k}_{1,2}$ and
$\tchb{n}{k}_{2,2}$ triangles are row-reversed Riordan arrays and it
was shown (in Corollary~37 of AE22) that the $\tchb{n}{k}_{m>2,2}$
triangles are not. From Theorem~35 of AE22, the
$\tchb{n}{k}_{m\ge2,t\ge3}$ triangles are not row-reversed Riordan
arrays since when tiling with $(1,m-1;t)$-combs and squares, the
filled-comb metatile contains more than one square if
$(m-1)(t-1)>1$. The same theorem tells us that, except for the $m=1$
cases, the triangles are also not Riordan arrays since there are
metatiles containing more than one comb.

There are a number of types of tiling that lead to common-node-free
digraphs that have only a few inner cycles. As far as we are aware,
Theorem~\ref{T:PCN} is the first result giving recursion relations for
a class of such cases. The theorem can be modified or generalized to
cope with a wider variety of classes and we will present these results
in future studies involving applications of tilings where instances of
such digraphs arise.

\section{Appendix: Recursion relations for 3-inner-cycle digraphs with a
  pseudo-common node} 

{\allowdisplaybreaks
For a digraph lacking a common node, we refer to an inner cycle that
can be represented as a single arc linking a node $\mathcal{E}$ to
itself as an \textit{errant loop} if the digraph would have a common
node $\mathcal{P}$ if the errant loop arc were removed. The node
$\mathcal{P}$ is then referred to as a \textit{pseudo-common
  node}. Evidently, $\mathcal{E}$ and $\mathcal{P}$ cannot be the same
node; if they were the same node, the original digraph would have a
true common node. For a digraph with an errant loop, a \textit{common
  circuit} is defined as two concatenated simple paths from the 0 node
to $\mathcal{P}$ and from $\mathcal{P}$ to the 0 node.  An outer
cycle, inner cycle, or common circuit is said to be \textit{plain} if
it does not include the errant loop node $\mathcal{E}$. See the proof
of Identity~\ref{I:rr25} for examples.

We use the $N=2$ case of the following lemma in the proof of
Theorem~\ref{T:PCN}.

\begin{lemma}\label{L:multinom}
For positive integers $j_0,j_1,\ldots,j_N$ where $N\ge2$,
\zmlg
\begin{multline*} 
\binom{j_1+\cdots+j_N}{j_1,\ldots,j_N}\binom{j_0+j_N-1}{j_0}
\!=\!
\sum_{r=1}^{N-1}\binom{j_1+\cdots+j_N-1}{j_1,\ldots,j_r-1,\ldots}
\Biggl(\!\binom{j_0+j_N-1}{j_0}-\binom{j_0+j_N-2}{j_0-1}\!\Biggr)\\
+\binom{j_1+\cdots+j_N}{j_1,\ldots,j_N}\binom{j_0+j_N-2}{j_0-1}
+\binom{j_1+\cdots+j_N-1}{j_1,\ldots,j_N-1}\binom{j_0+j_N-2}{j_0}.
\end{multline*} 
\end{lemma}
\begin{proof}
Using the result for multinomial coefficients that
\[
\binom{j_1+\cdots+j_N}{j_1,\ldots,j_N}
=\sum_{r=1}^{N}\binom{j_1+\cdots+j_N-1}{j_1,\ldots,j_r-1,\ldots,j_N},
\]
we have
\begin{align*} 
&\binom{j_1+\cdots+j_N}{j_1,\ldots,j_N}\binom{j_0+j_N-1}{j_0}\\
&\qquad\qquad=\Biggl(\sum_{r=1}^{N-1}\binom{j_1+\cdots+j_N-1}{j_1,\ldots,j_r-1,\ldots} 
+\binom{j_1+\cdots+j_N-1}{j_1,\ldots,j_N-1}\Biggr)\binom{j_0+j_N-1}{j_0}\\
&\qquad\qquad=\sum_{r=1}^{N-1}\binom{j_1+\cdots+j_N-1}{j_1,\ldots,j_r-1,\ldots}\binom{j_0+j_N-1}{j_0}\\
&\qquad\qquad\qquad+\binom{j_1+\cdots+j_N-1}{j_1,\ldots,j_N-1}
\Biggl(\binom{j_0+j_N-2}{j_0}+\binom{j_0+j_N-1}{j_0-1}\Biggr)\\
&\qquad\qquad=\sum_{r=1}^{N-1}\binom{j_1+\cdots+j_N-1}{j_1,\ldots,j_r-1,\ldots}\binom{j_0+j_N-1}{j_0}+\binom{j_1+\cdots+j_N-1}{j_1,\ldots,j_N-1}\binom{j_0+j_N-2}{j_0}\\
&\qquad\qquad\qquad+\Biggl( 
\binom{j_1+\cdots+j_N}{j_1,\ldots,j_N}
-\sum_{r=1}^{N-1}\binom{j_1+\cdots+j_N-1}{j_1,\ldots,j_r-1,\ldots,j_N}
\Biggr)\binom{j_0+j_N-2}{j_0-1},
\end{align*} 
which gives the required result on rearranging.
\end{proof}

\begin{theorem}\label{T:PCN}
For a digraph with an errant loop of length $L_0$ containing $K_0$ 
combs, a plain inner cycle of length $L_1$ containing $K_1$ combs,
a non-plain inner cycle of length $L_2$ containing $K_2$ combs, and
outer cycles that are all plain and have length $l_{\mathrm{o}i}$ and
contain $k_{\mathrm{o}i}$ combs for $i=1,\ldots,N\rb{o}$, let
$l_{\mathrm{c}i}$ be the length of the $i$-th common circuit
and let $k_{\mathrm{c}i}$ be the number of combs it contains
($i=1,\ldots,N\rb{c}$), and let $l_{\mathrm{pc}i}$ be the length of
the $i$-th plain common circuit and let $k_{\mathrm{pc}i}$ be the
number of combs it contains ($i=1,\ldots,N\rb{pc}$). Then for all
integers $n$ and $k$,
\zmlg
\begin{align}\label{e:PCNBn}
B_n&=\delta_{n,0}
+\sum_{r=0}^2 (B_{n-L_r}-\delta_{n,L_r})
+\delta_{n,L_0+L_1}-B_{n-L_0-L_1}\nonumber\\
&\qquad+\sum_{i=1}^{N\rb{o}}
\biggl(B_{n-l_{\mathrm{o}i}}
+B_{n-l_{\mathrm{o}i}-L_0-L_1}
-\sum_{r=0}^NB_{n-l_{\mathrm{o}i}-L_r}\biggr)
+\sum_{i=1}^{N\rb{c}} B_{n-l_{\mathrm{c}i}}
-\sum_{i=1}^{N\rb{pc}} B_{n-l_{\mathrm{pc}i}-L_0},\\
B_{n,k}&=\delta_{n,0}\delta_{k,0}
+\sum_{r=0}^2 (B_{n-L_r,k-K_r}-\delta_{n,L_r}\delta_{k,K_r})
+ 
\delta_{n,L_0+L_1}\delta_{k,K_0+K_1}-B_{n-L_0-L_1,k-K_0-K_1}\nonumber\\
&\qquad+\sum_{i=1}^{N\rb{o}}
\biggl(B_{n-l_{\mathrm{o}i},k-k_{\mathrm{o}i}}
+B_{n-l_{\mathrm{o}i}-L_0-L_1,k-k_{\mathrm{o}i}-K_0-K_1}
-\sum_{r=0}^NB_{n-l_{\mathrm{o}i}-L_r,k-k_{\mathrm{o}i}-K_r}\biggr)\nonumber\\
&\qquad+\sum_{i=1}^{N\rb{c}} B_{n-l_{\mathrm{c}i},k-k_{\mathrm{c}i}}
-\sum_{i=1}^{N\rb{pc}} B_{n-l_{\mathrm{pc}i}-L_0,k-k_{\mathrm{pc}i}-K_0},
\label{e:PCNBnk}
\end{align} 
where $B_{n<0}=B_{n,k<0}=B_{n<k,k}=0$. If the lengths of the cycles
and circuits are calculated as the number of tiles (the total contribution
made to the number of cells occupied) then
$B_n$ is the number of $n$-tile tilings (the number of tilings of an
$n$-board) and $B_{n,k}$ is the number of such tilings that use $k$ combs.

\end{theorem}
\begin{proof}
To keep the algebra looking as simple as possible while retaining the
essentials at the heart of the proof, we just prove the formula for
$B_n$ when there is a single outer cycle, one plain common circuit,
and one non-plain common circuit. Their respective lengths
are $l\rb{o}$, $l\rb{pc}$, and $l\rb{npc}$. It is straightforward to
modify the proof we give here to include the sums over outer cycles
and common circuits. The proof of \eqref{e:PCNBnk} is entirely analogous.

Conditioning on the final metatile gives
\zmlg
\begin{multline}\label{e:condPCN} 
B_n=\delta_{n,0}+B_{n-l\rb{o}}
+\sum_{j_1\ge0}B_{n-l\rb{pc}-j_1L_1}
+\sum_{\substack{j_0,j_1\ge0,\\j_2\ge1}}\!\!\!
\binom{j_1+j_2}{j_1}\binom{j_0+j_2-1}{j_0}B_{n-l\rb{pc}-j_0L_0-j_1L_1-j_2L_2}\\
+\sum_{e,j_1\ge0}B_{n-l\rb{npc}-eL_0-j_1L_1}
+\sum_{\substack{e,j_0,j_1\ge0,\\j_2\ge1}}\!\!\!
\binom{j_1+j_2}{j_1}\binom{j_0+j_2-1}{j_0}B_{n-l\rb{npc}-(j_0+e)L_0-j_1L_1-j_2L_2}
\end{multline} 
with $B_{n<0}=0$. We now explain the origin of the four sums in
\eqref{e:condPCN} while referring to the digraph in Fig.~\ref{f:25}
(taking the 010101 node as $\mathcal{P}$) for examples of
metatiles. The first sum is from metatiles obtained by taking the
first part of the plain common circuit to $\mathcal{P}$, then
following the plain inner cycle $j_1$ times, and then returning to the
0 node via the second half of the plain common circuit (e.g.,
$CSC^{2j_1}CS$ is the symbolic representation of the metatiles
corresponding to the terms in the sum). 

The second sum corresponds to metatiles with the same start and end as
with the first sum but on reaching $\mathcal{P}$ the plain and
non-plain inner cycles are executed $j_1$ and $j_2$ times,
respectively, in any order but the non-plain inner cycle is executed
at least once. The number of ways of choosing the order is
$\tbinom{j_1+j_2}{j_1}$. The errant loop is also traversed a total of
$j_0$ times. Each time the path reaches $\mathcal{E}$ (during an
execution of the non-plain inner cycle), it can detour and traverse
the errant loop any number of times. This is the origin of the
$\tbinom{j_0+j_2-1}{j_0}$ factor which has $j_2-1$ rather than $j_2$
as the non-plain inner cycle must be started before the errant loop can
be traversed. E.g., the metatiles corresponding to the terms in the
sum when $j_0=j_1=j_2=1$ are $CS\{C^2SCS,SCSC^2\}CS$.

The third and fourth sums are analogous to the first and second but
$\mathcal{P}$ is reached via the first half of the non-plain common
circuit, and after the inner cycles have been traversed $j_r$ times
(with $r=1$ in the third sum and $r=0,1,2$ in the fourth), the 0 node
is returned to via the second half of the non-plain common circuit
but the errant loop is executed an extra $e$ times when the path reaches
$\mathcal{E}$. E.g., the metatiles corresponding to the terms in the
third sum are $CSC^{2j_1}SC^eS^2$, and in the fourth sum when
$j_0=j_1=j_2=1$ they are $CS\{C^2SCS,SCSC^2\}SC^eS^2$. 

Representing \eqref{e:condPCN} by $E(n)$, we write down
\[
E(n)-E(n-L_0)-E(n-L_1)+E(n-L_0-L_1)-E(n-L_2)
\]
and re-index the sums so that, where possible, the $B_{n-\alpha}$
inside the sums for any $\alpha$ appear the same as for $E(n)$ (e.g.,
$\sum_{j_1\ge0}B_{n-L_1-l\rb{pc}-j_1L_1}=\sum_{j_1\ge1}B_{n-l\rb{pc}-j_1L_1}$).
This leaves
\begin{multline}\label{e:subPCN} 
B_n-\sum_{r=0}^2B_{n-L_r}+B_{n-L_0-L_1}=\delta_{n,0}+B_{n-l\rb{o}}
-\sum_{r=0}^2(\delta_{n,L_r}+B_{n-l\rb{o}-L_r})+\delta_{n,L_0+L_1}
+B_{n-l\rb{o}-L_0-L_1}\\
+\sum_{j_1\ge0}\beta_{j_1L_1}-\sum_{j_1\ge1}\beta_{j_1L_1}
-\sum_{j_1\ge0}\beta_{L_0+j_1L_1}+\sum_{j_1\ge1}\beta_{L_0+j_1L_1}
-\sum_{j_1\ge0}\beta_{j_1L_1+L_2}\\
+\sum_{e,j_1\ge0}\hat{\beta}_{j_1L_1}
-\sum_{\substack{e\ge0,\\j_1\ge1}}\hat{\beta}_{j_1L_1}
-\sum_{\substack{e\ge1,\\j_1\ge0}}\hat{\beta}_{j_1L_1}
+\sum_{e,j_1\ge1}\hat{\beta}_{j_1L_1}
-\sum_{e,j_1\ge0}\hat{\beta}_{j_1L_1+L_2}\\
+\sum_{\substack{j_0,j_1\ge0,\\j_2\ge1}}
\binom{j_1+j_2}{j_1}\binom{j_0+j_2-1}{j_0}\beta_\lambda
-\sum_{\substack{j_0\ge0,\\j_1,j_2\ge1}}
\binom{j_1+j_2-1}{j_1-1}\binom{j_0+j_2-1}{j_0}\beta_\lambda\\
-\sum_{\substack{j_0,j_2\ge1,\\j_1\ge0}}
\binom{j_1+j_2}{j_1}\binom{j_0+j_2-2}{j_0-1}\beta_\lambda
+\sum_{j_0,j_1,j_2\ge1}
\binom{j_1+j_2-1}{j_1-1}\binom{j_0+j_2-2}{j_0-1}\beta_\lambda\\
-\sum_{\substack{j_0,j_1\ge0,\\j_2\ge2}}
\binom{j_1+j_2-1}{j_1}\binom{j_0+j_2-2}{j_0}\beta_\lambda\\
+\text{the above 3 lines with $\beta_\lambda$ replaced by
  $\hat{\beta}_\lambda$ and also summed over all $e\ge0$},
\end{multline} 
where $\beta_a=B_{n-l\rb{pc}-a}$,
$\hat{\beta}_a=B_{n-l\rb{npc}-eL_0-a}$, and
$\lambda=j_0L_0+j_1L_1+j_2L_2$. On rearranging \eqref{e:subPCN} it is
immediately apparent where all but the last two sums in \eqref{e:PCNBn}
come from.  The first two sums in the second line of \eqref{e:subPCN}
reduce to $\beta_0=B_{n-l\rb{pc}}$. The next two sums reduce to
$-\beta_{L_0}=-B_{n-l\rb{pc}-L_0}$ which accounts for the final sum in
\eqref{e:PCNBn}. The first four sums in the third
line of \eqref{e:subPCN} reduce to $B_{n-l\rb{npc}}$ which when added
to the $B_{n-l\rb{pc}}$ accounts for the penultimate sum in \eqref{e:PCNBn}.

We now complete the proof by showing that the remaining terms in
\eqref{e:subPCN} cancel out.  We regroup terms in each of the sums in
the fourth, fifth, and sixth lines in \eqref{e:subPCN} to give
\begin{subequations}
\label{e:456}
\begin{align} 
&\sum_{\substack{j_0,j_1\ge0,\\j_2\ge1}}\!\!\!
\binom{j_1+j_2}{j_1}\binom{j_0+j_2-1}{j_0}\beta_\lambda=\!\!
\sum_{\substack{j_0,j_1\ge1,\\j_2\ge2}}\!\!\!
\binom{j_1+j_2}{j_1}\binom{j_0+j_2-1}{j_0}\beta_\lambda\nonumber\\
&\qquad+\sum_{j_0,j_1\ge1}\!\!\!
\binom{j_1+1}{j_1}\beta_{j_0L_0+j_1L_1+L_2}
+\sum_{\substack{j_0\ge1,\\j_2\ge2}}\!
\binom{j_0+j_2-1}{j_0}\beta_{j_0L_0+j_2L_2}
+\sum_{\substack{j_1\ge1,\\j_2\ge2}}\!
\binom{j_1+j_2}{j_1}\beta_{j_1L_1+j_2L_2}\nonumber\\
&\qquad+\sum_{j_0\ge1}\beta_{j_0L_0+L_2}
+\sum_{j_1\ge1}\binom{j_1+1}{j_1}\beta_{j_1L_1+L_2}
+\sum_{j_2\ge2}\beta_{j_2L_2}
+\beta_{L_2},
\label{e:a} \\
&\sum_{\substack{j_0\ge0,\\j_1,j_2\ge1}}\!\!\!
\binom{j_1+j_2-1}{j_1-1}\binom{j_0+j_2-1}{j_0}\beta_\lambda=\!\!
\sum_{\substack{j_0,j_1\ge1,\\j_2\ge2}}\!\!\!
\binom{j_1+j_2-1}{j_1-1}\binom{j_0+j_2-1}{j_0}\beta_\lambda
\nonumber\\
&\qquad+\sum_{j_0,j_1\ge1}\!\!\!
\binom{j_1}{j_1-1}\beta_{j_0L_0+j_1L_1+L_2}
+\sum_{\substack{j_1\ge1,\\j_2\ge2}}\!
\binom{j_1+j_2-1}{j_1-1}\beta_{j_1L_1+j_2L_2}
+\sum_{j_1\ge1}\binom{j_1}{j_1-1}\beta_{j_1L_1+L_2},
\label{e:b} \\
&\sum_{\substack{j_0,j_2\ge1,\\j_1\ge0}}\!\!\!
\binom{j_1+j_2}{j_1}\binom{j_0+j_2-2}{j_0-1}\beta_\lambda=\!\!
\sum_{\substack{j_0,j_1\ge1,\\j_2\ge2}}\!\!\!
\binom{j_1+j_2}{j_1}\binom{j_0+j_2-2}{j_0-1}\beta_\lambda
\nonumber\\
&\qquad+\sum_{j_0,j_1\ge1}\!\!\!
\binom{j_1+1}{j_1}\beta_{j_0L_0+j_1L_1+L_2}
+\sum_{\substack{j_0\ge1,\\j_2\ge2}}\!
\binom{j_0+j_2-1}{j_0-1}\beta_{j_0L_0+j_2L_2}
+\sum_{j_0\ge1}\beta_{j_0L_0+L_2},
\label{e:c} \\
&\sum_{j_0,j_1,j_2\ge1}\!\!\!
\binom{j_1+j_2-1}{j_1-1}\binom{j_0+j_2-2}{j_0-1}\beta_\lambda=\!\!
\sum_{\substack{j_0,j_1\ge1,\\j_2\ge2}}\!\!\!
\binom{j_1+j_2-1}{j_1-1}\binom{j_0+j_2-2}{j_0-1}\beta_\lambda\nonumber\\
&\qquad+\sum_{j_0,j_1\ge1}\!\!\!
\binom{j_1}{j_1-1}\beta_{j_0L_0+j_1L_1+L_2},
\label{e:d} \\
&\sum_{\substack{j_0,j_1\ge0,\\j_2\ge2}}\!\!\!
\binom{j_1+j_2-1}{j_1}\binom{j_0+j_2-2}{j_0}\beta_\lambda=\!\!
\sum_{\substack{j_0,j_1\ge1,\\j_2\ge2}}\!\!\!
\binom{j_1+j_2-1}{j_1}\binom{j_0+j_2-2}{j_0}\beta_\lambda\nonumber\\
&\qquad+\sum_{\substack{j_0\ge1,\\j_2\ge2}}\!
\binom{j_0+j_2-1}{j_0}\beta_{j_0L_0+j_2L_2}
+\sum_{\substack{j_1\ge1,\\j_2\ge2}}\!
\binom{j_1+j_2-1}{j_1}\beta_{j_1L_1+j_2L_2}
+\sum_{j_2\ge2}\beta_{j_2L_2}.
\label{e:e} 
\end{align} 
\end{subequations}
We denote the $p$-th sum (or term) on the right-hand side of
(\ref{e:456}$x$) by $x_p$ where $x$ is a--e. Then 
$\mathrm{a}_1-\mathrm{b}_1-\mathrm{c}_1+\mathrm{d}_1-\mathrm{e}_1=0$ by
virtue of Lemma~\ref{L:multinom}, 
$\mathrm{a}_2$ cancels $\mathrm{c}_2$, 
$\mathrm{a}_3$ cancels $\mathrm{c}_3+\mathrm{e}_2$,  
$\mathrm{a}_4$ cancels $\mathrm{b}_3+\mathrm{e}_3$, 
$\mathrm{a}_5$ cancels $\mathrm{c}_4$,
$\mathrm{a}_6-\mathrm{b}_4+\mathrm{a}_8=\sum_{j_1\ge0}\beta_{j_1L_1+L_2}$
and therefore cancels the last sum in the second line of \eqref{e:subPCN},
$\mathrm{a}_7$ cancels $\mathrm{e}_4$, and
$\mathrm{b}_2$ cancels $\mathrm{d}_2$. The simplification works in the
same way for the terms represented by the last line of
\eqref{e:subPCN}. Denoting sums or terms in 
the corresponding set of equations by $\hat{x}_p$, 
$\hat{\mathrm{a}}_6-\hat{\mathrm{b}}_4+\hat{\mathrm{a}}_8
=\sum_{e,j_1\ge0}\hat{\beta}_{j_1L_1+L_2}$
and therefore cancels the last sum in the third line of \eqref{e:subPCN}.

The proof of \eqref{e:PCNBnk} proceeds in an analogous way. Again
considering the case where there is a single outer cycle (with
$k\rb{o}$ combs), a plain common circuit (with $k\rb{pc}$ combs), and
a non-plain common circuit (with $k\rb{npc}$ combs),
conditioning on
the final metatile gives
\begin{multline}\label{e:condPCNnk} 
B_{n,k}=\delta_{n,0}\delta_{k,0}+B_{n-l\rb{o},k-k\rb{o}}
+\sum_{j_1\ge0}B_{n-l\rb{pc}-j_1L_1,n-k\rb{pc}-j_1K_1}\\
+\sum_{\substack{j_0,j_1\ge0,\\j_2\ge1}}\!\!\!
\binom{j_1+j_2}{j_1}\binom{j_0+j_2-1}{j_0}
B_{n-l\rb{pc}-\lambda,k-k\rb{pc}-\kappa}
+\sum_{e,j_1\ge0}B_{n-l\rb{npc}-eL_0-j_1L_1,k-k\rb{npc}-eK_0-j_1K_1}\\
+\sum_{\substack{e,j_0,j_1\ge0,\\j_2\ge1}}\!\!\!
\binom{j_1+j_2}{j_1}\binom{j_0+j_2-1}{j_0}B_{n-l\rb{npc}-eL_0-\lambda,k-k\rb{npc}-eK_0-\kappa}
\end{multline} 
with $B_{n,k>n}=B_{n,k<0}=0$ and where $\kappa=j_0K_0+j_1K_1+j_2K_2$.
Denoting \eqref{e:condPCNnk} by $E(n,k)$, writing down 
\[
E(n,k)-E(n-L_0,k-K_0)-E(n-L_1,k-K_1)+E(n-L_0-L_1,k-K_0-K_1)-E(n-L_2,k-K_2),
\]
and then proceeding in the same way as for the proof of
\eqref{e:PCNBn} gives the required result. 
\end{proof}
}

\bigskip
\hrule
\bigskip

\noindent 2010 {\it Mathematics Subject Classification}:
Primary 11B39; Secondary 05A19, 05A15.

\noindent \emph{Keywords}:
combinatorial proof, combinatorial identity, $n$-tiling,
Pascal-like triangle, directed pseudograph, Fibonacci polynomial,
restricted combination

\bigskip
\hrule
\bigskip

\sloppy
\noindent (Concerned with sequences 
\seqnum{A000045},
\seqnum{A000930},
\seqnum{A003269},
\seqnum{A003520},
\seqnum{A005578},
\seqnum{A005708},
\seqnum{A005709},
\seqnum{A005710},
\seqnum{A007318}, 
\seqnum{A011782},
\seqnum{A031923},
\seqnum{A099163},
\seqnum{A224808},
\seqnum{A224809},
\seqnum{A224811},
\seqnum{A350110}, 
\seqnum{A350111}, 
\seqnum{A350112},
\seqnum{A354665},
\seqnum{A354666},
\seqnum{A354667}, and 
\seqnum{A354668}) 

\bigskip
\hrule
\bigskip

\end{document}